\documentclass[11pt]{article}
\usepackage{amsmath}
\usepackage{amsfonts}
\usepackage{amsthm}
\usepackage{amssymb, setspace}
\usepackage{color,graphicx,epsfig,geometry,hyperref,fancyhdr}
\usepackage[T1]{fontenc}
\usepackage{float}
\usepackage{subfig}
\usepackage{commath}
\usepackage{bbm}
\usepackage{mathrsfs,fleqn}
\usepackage{mathtools}
\newtheorem{theorem}{Theorem}[section]

\newtheorem{lemma}{Lemma}[section]

\newcommand{\R}{\mathbb{R}}
\newcommand{\C}{\mathbb{C}}

\newcommand{\brac}[1]{\left\{#1\right\}}

\graphicspath{{paper-pics/}}

\usepackage{xcolor,cancel}%

\begin{document}

\begin{flushleft}
\Large 
\noindent{\bf \Large Reconstruction of extended regions in EIT with a generalized Robin transmission condition}
\end{flushleft}

\vspace{0.2in}

{\bf  \large Govanni Granados, Isaac Harris, and Heejin Lee}\\
\indent {\small Department of Mathematics, Purdue University, West Lafayette, IN 47907 }\\
\indent {\small Email:  \texttt{ggranad@purdue.edu}, \texttt{harri814@purdue.edu} and   \texttt{lee4485@purdue.edu }}\\


\begin{abstract}
\noindent We consider an inverse shape problem coming from electrical impedance tomography with a generalized Robin transmission condition. We will derive an algorithm in order to detect whether two materials that should be in contact are separated or delaminated. More precisely, we assume that the undamaged material or background state is known and shares an interface or boundary with the damaged subregion. The Robin transmission condition on this boundary asymptotically models delamination. We assume that the DtN operator is given from measuring the current on the surface of the material from an imposed voltage. We show that this mapping uniquely recovers the boundary parameters. Furthermore, using this electrostatic Cauchy data as physical measurements, we can determine if all of the coefficients from the Robin transmission condition are real-valued or complex-valued. We study these two cases separately and show that the regularized factorization method can be used to detect whether delamination has occurred and recover the damaged subregion. Numerical examples will be presented for both cases in two dimensions in the unit circle. 
\end{abstract}

\noindent {\bf Keywords}: Electrical Impedance Tomography $\cdot$ Regularized Factorization Method $\cdot$ Shape Reconstruction \\

\noindent {\bf MSC}:  35J05, 35J25

\section{Introduction}
In this paper, we consider an inverse shape problem in electrostatic imaging. The problem is motivated by electrical impedance tomography(EIT) where the goal is to reconstruct unknown interior defects from the measured electrostatic data on the surface of an object. We apply a Qualitative Method to recover said regions where the knowledge of the solution to a boundary value problem is used in their detection. We are interested in the scenario of reconstructing a subregion where a generalized Robin transmission condition is imposed. The generalized Robin condition we consider models the delamination of a subregion within a known material. This generalized Robin condition has that there is a jump in the normal derivative of the electrostatic potential across the delaminated subregion's boundary (i.e. the boundary between the healthy region and the unknown defective subregion) and is quasi-proportional to the electrostatic potential itself. This is a generalization to the boundary condition of the inverse shape problem studied in \cite{EIT-granados1} and \cite{harris2}. In non-invasive and non-destructive testing, one wishes to recover the location of all possible subregions of interest in a given material using data on the material's surface. For other works on non-destructive testing based on electromagnetic imaging we refer to \cite{EM-cakoni, deng1, kharkovsky1}. For related problems in medical imaging we refer to \cite{franchois1, tamori1}. The delamination corresponds to defects in the material that one wishes to recover without corrupting the integrity of a possibly healthy material. See \cite{eit-review,eit-review-amend,EIT-cheney, MUSIC-Hanke2,mueller-book} for more discussion on the theory and applications of EIT. 

We will derive an algorithm for recovering the defective subregions with little a priori information. One of the strengths of applying Qualitative Methods is that one does not need to know the number of defective regions or have an estimate for the boundary parameters. On the contrary, many iterative methods are locally convergent i.e. they require a ``good'' initial estimate for the unknown region and/or parameters to insure convergence to the solution of the inverse problem. Qualitative Methods allow one to reconstruct regions by deriving an `indicator' function from the measured data operator. This idea was first introduced in \cite{colton1} and is done by connecting the region of interest to the range of the measured data operator. We assume that voltage is applied to the known exterior boundary of the material and the induced current is measured also on the exterior boundary. Thus, we assume that we have full knowledge of the Dirichlet-to-Neumann(DtN) mapping on the exterior boundary for Laplace's equation in the domain with a delaminated subregions. In \cite{eit-transmission1,eit-transmission2,JIIP} the authors studied the inverse parameter problem for the EIT problem with with a Robin transmission condition. In the aforementioned papers, the authors studied the uniqueness, stability and numerical reconstruction for the inverse parameter problem using the Neumann-to-Dirichlet mapping. In \cite{CHK-gibc}, the authors analyzed this problem in $\R^2$ via a system of non-linear boundary integral equations. Also, see \cite{FM-GIBC} for the factorization method applied to inverse obstacle scattering with a similar boundary condition. Here, we study the inverse shape problem and prove that the DtN mapping uniquely determines the boundary coefficients as well as uniquely recovers the region of interest.

To this end, we consider the regularized factorization method, which is a type of sampling method, for solving the inverse shape problem. This regularized variant of the factorization method was initially studied in \cite{harris1} for a similar problem coming from diffuse optical tomography and see \cite{regfm2} for stability. This method is based on the analysis in \cite{arens,GLSM,EIT-FM,kirschbook}. The analysis we present here works in both $\R^2$ or $\R^3$ making these methods robust in their applications. By connecting the region of interest to the range of the measured DtN mapping, one can characterize the unknown region $D$ by the singular-value decomposition of the measured data operator. This makes the numerical implementation computationally inexpensive, whereas an iterative method would require solving multiple adjoint problems at each iteration.

The rest of the paper is structured as follows. In Section \ref{dp-ip}, we rigorously formulate the direct and inverse problem under consideration. We will use a variational method to prove well-posedness for the direct problem and derive the appropriate functional settings. We also define the current-gap operator $(\Lambda - \Lambda_0)$ that will be used to derive a regularized factorization method to recover the damaged subregion. In Section \ref{uniqueness-section}, we discuss the uniqueness of the inverse impedance problem by showing the injectivity of the mapping of the boundary coefficients to the DtN operator. We continue in Section \ref{complex-section} and Section \ref{real-section} where we analyze $(\Lambda - \Lambda_0)$ to derive a suitable factorization for the case when the boundary parameters are complex-valued and real-valued, respectively. This will allow us to develop a reconstruction algorithm, which will depend on range-based identities, to recover $D$. In Section \ref{numerical-validation}, numerical examples are presented in $\R^2$ for the unit circle to validate the analysis of the reconstruction algorithm. Finally, we give a brief summary and conclusion of the results in Section \ref{conclusion}.

\section{The Direct and Inverse Problem}\label{dp-ip}
We begin by considering the direct problem associated with the electrostatic imaging of a defective region with a generalized Robin transmission condition on its boundary. Assume that $ \Omega \subset \mathbb{R}^d$ is a simply connected open set with $\mathcal{C}^2$--boundary(or polygonal with no reentrant corners) $\partial \Omega$. Let $D \subset \Omega$ be a (possibly multiple) connected open set with class $\mathcal{C}^2$--boundary(or polygonal with no reentrant corners) $\partial D$. Throughout the paper, we will assume that $\text{dist}(\partial \Omega , \overline{D}) > 0$. For the material with defective region(s), we define $u \in H^1 (\Omega)$ as the solution to
\begin{equation}\label{gen-pde}
- \Delta u = 0 \enspace \text{in} \enspace \Omega \textbackslash \partial D \quad \text{with} \quad u \big|_{\partial \Omega} = f \quad \text{and} \quad [\![\partial_\nu u ]\!] \big|_{\partial D} = \mathcal{B}(u) \enspace \text{on} \enspace \partial D
\end{equation}
where
$$ [\![\partial_\nu u ]\!] \big|_{\partial D} := ( \partial_{\nu} u^{+} - \partial_{\nu} u^{-}) \big|_{\partial D} $$ 
for a given $f \in H^{1/2} ( \partial \Omega )$.
For the rest of the paper, we let $\nu$ denote the unit outward normal on the boundaries $\partial D$ and $\partial \Omega$. The `+' notation represents the trace taken from $\Omega \setminus \overline{D}$ and the `$-$' notation represents the trace taken from $D$. Here, the function $u$ is the electrostatic potential for the defective material satisfies the boundary condition with the general Laplace-Beltrami boundary operator 
\begin{equation}\label{lap-bel}
\mathcal{B}(u) = - \nabla_{\partial D} \cdot \mu \nabla_{\partial D}u + \gamma u.
\end{equation}
In the $\mathbb{R}^2$ case, the operator $ \nabla_{\partial D} \cdot \mu \nabla_{\partial D}$ is replaced by the operator $\frac{\text{d}}{\text{d}s} \mu \frac{\text{d}}{\text{d}s}$ where ${\text{d}} / \text{d}s$ is the tangential derivative and $s$ is the arc-length. The generalized Robin condition in \eqref{gen-pde} models the delamination of the defective region $D \subset \Omega$ on its boundary $\partial D$ and states that the jump in current across this boundary is quasi-proportional to the electrostatic potential $u$. The analysis in this paper holds for dimensions $d=2$ and $d=3$. 

We consider the two cases when the boundary parameters $\gamma \in L^{\infty}(\partial D)$ and $\mu_j \in L^{\infty}(\partial D)$ are complex- and real-value where $j=1,\cdots,d-1$. Due to the generalized Robin transmission condition \eqref{lap-bel}, we consider finding the solution $u \in \widetilde{H}^{1}(\Omega)$ to \eqref{gen-pde} for a given $f \in H^{1/2}(\partial \Omega)$ where the solution space is the Hilbert space defined as
$$ \widetilde{H}^{1}(\Omega) = \brac{ u \in H^{1}(\Omega) \quad \text{such that} \quad u \big \rvert_{\partial D} \in H^{1}(\partial D) }$$
equipped with the norm
$$\norm{\varphi}^{2}_{\widetilde{H}^{1}(\Omega)} = \norm{\varphi}^{2}_{H^{1}(\Omega)} + \norm{\varphi}^{2}_{H^{1}(\partial D)}. $$
Since we assume that $u \in \widetilde{H}^{1}(\Omega) \subset H^1 (\Omega)$, it is known that $ [\![u ]\!]  \big|_{\partial D} = 0$. This comes from the fact that any function in $H^1 (\Omega)$ has equal interior trace `$-$' and exterior trace `+' on any subdomain of $\Omega$. We begin by showing that the boundary value problem \eqref{gen-pde} is well-posed for the case when the parameters $\gamma \in L^{\infty} ( \partial D, \mathbb{C})$ and $\mu \in L^{\infty} ( \partial D, \mathbb{C}^{(d-1) \times (d-1)})$ for any given $f \in H^{1/2} ( \partial \Omega)$. For analytical purposes of well-posedness of the direct problem and the analysis of the inverse problem of this case in Section 3, we assume that there exist positive constants $\zeta_1 , \zeta_2 > 0$ such that the real- and imaginary-part of the coefficient $\gamma$ satisfies
\begin{equation}\label{gamma-bounds}
\text{Re}(\gamma) \geq \zeta_1 > 0 \quad \text{and} \quad - \text{Im}(\gamma) \geq \zeta_2 > 0
\end{equation}
for almost every $x \in \partial D$. We also assume that the real- and imaginary-parts of $\mu$ are Hermitian definite matrices where there exist positive constants $\beta_1 , \beta_2 > 0$ such that
such that
\begin{equation}\label{mu-bounds}
 \overline{\xi} \cdot \text{Re}(\mu ) (x) \xi \geq \beta_1 |\xi|^2 \quad \text{and} \quad - \overline{\xi} \cdot \text{Im}(\mu ) (x) \xi \geq \beta_2 |\xi|^2  > 0
\end{equation}
for all $\xi \in \mathbb{C}^{d-1} \setminus \{0\}$ for almost every $x \in \partial D$. We now consider Green's 1st Theorem on the region $\Omega \backslash \overline{D}$ 
$$\int_{\Omega \backslash D} \nabla u \cdot \nabla \overline{\varphi} \, \text{d}x = \int_{\partial \Omega}  \overline{\varphi}  \partial_{\nu} u \, \text{d}s - \int_{D}  \overline{\varphi} \partial_{\nu} u^{+} \, \text{d}s $$
as well as Green's 1st Theorem on the region $D$
$$ \int_{D} \nabla u \cdot \nabla \overline{\varphi} \, \text{d}x = \int_{\partial D}  \overline{\varphi} \partial_{\nu} u^{-} \, \text{d}s $$ 
for any test function $\varphi \in \widetilde{H}^1 (\Omega)$. The variational formulation for \eqref{gen-pde} is given by  adding these two equations
\begin{equation}\label{gen-vf}
\int_{\Omega} \nabla u \cdot \nabla \overline{\varphi} \, \text{d}x =  \int_{\partial \Omega} \overline{\varphi} \partial_{\nu} u  \, \text{d}s +\int_{\partial D} \overline{\varphi}  \nabla_{\partial D} \cdot \mu \nabla_{\partial D} u  \, \text{d}s -  \int_{\partial D} \overline{\varphi} \gamma u  \, \text{d}s
\end{equation}
where we have used the generalized Robin transmission condition on $\partial D$ \eqref{lap-bel}. In order to proceed, we let $u_0 \in H^{1} ( \Omega)$ be the harmonic lifting of the Dirichlet data such that  
\begin{equation}\label{hl}  
- \Delta u_0 = 0 \enspace \text{in} \enspace \Omega \quad \text{with} \quad u_0 \big|_{\partial \Omega} = f.
\end{equation}
We make the ansatz that the solution can be written as $u = v + u_0$ with the function $v \in \widetilde{H}_{0}^{1} (\Omega)$ where we define the space as 
$$\widetilde{H}_{0}^{1} (\Omega) = \brac{ \varphi \in \widetilde{H}^{1} (\Omega) \quad \text{such that} \quad \varphi \big \rvert_{\partial \Omega} = 0 }$$
with the same norm as $\widetilde{H}^{1} (\Omega)$. Thus, the variational formulation of \eqref{gen-pde} with respect to $v$ is given by
\begin{equation} \label{ses}
A(v , \varphi) = - A(u_0 , \varphi) \quad  \text{for all} \quad \varphi \in \widetilde{H}_{0}^{1} (\Omega)
\end{equation}
where the sesquilinear form $A(\cdot , \cdot): \widetilde{H}_{0}^{1} (\Omega) \times \widetilde{H}_{0}^{1} ( \Omega ) \mapsto \mathbb{C}$ is given by 
$$ A(v , \varphi) = \int_{\Omega} \nabla v \cdot \nabla \overline{\varphi} \, \text{d}x + \int_{\partial D} \mu \nabla_{\partial D} v \cdot \nabla_{\partial D} \overline{\varphi} \, \text{d}s + \int_{\partial D} \gamma \, v \, \overline{\varphi} \, \text{d}s.$$
It is clear that the sesqulinear form is bounded whereas the coercivity on $\widetilde{H}^1_0(\Omega)$ can be shown by the assumptions on $\mu$ and $\gamma$ as well as the Poincar\'{e} inequality. We also have that $A(u_0 , \varphi)$ is a conjugate linear and bounded functional acting on  $\widetilde{H}_{0}^{1} (\Omega)$ and using the Trace Theorem just as in \cite{harris2} we have that 
$$| A(u_0 , \varphi ) | \leq C \|{f}\|_{H^{1/2} (\partial \Omega)} \|{\varphi}\|_{\widetilde{H}^{1} (\Omega)}.$$ By the Lax-Milgram lemma, there is a unique solution $v$ to \eqref{ses} satisfying 
$$\|{v}\|_{\widetilde{H}^{1} (\Omega)} \leq C \|{f}\|_{H^{1/2} (\partial \Omega)}.$$ 
Using the sesquilinear form $A(\cdot \,  , \cdot)$, we can show that the solution $u$ for equation \eqref{gen-pde} is unique just as in \cite{harris2}, which implies that equation \eqref{gen-pde} is well-posed. The above analysis gives the following result.

\begin{theorem}\label{gen-soloperator}
The solution operator corresponding to the boundary value problem \eqref{gen-pde} $f \mapsto u$ is a bounded linear mapping from $H^{1/2} ( \partial \Omega)$ to $\widetilde{H}^{1}(\Omega)$.
\end{theorem}

We assume that the voltage $f$ is applied to the outer boundary $\partial \Omega$ and the measured data is given by the current $\partial_{\nu}u$. From the knowledge of the measured currents, we wish to derive a qualitative sampling algorithm to determine the defective region $D$ without the knowledge of the boundary parameters $\mu$ and $\gamma$ and with little to no prior knowledge on the number of regions. To this end, we define the data operator that will be studied in the following sections to derive our algorithms. Note that the function $u_0$ is the electrostatic potential for the healthy material and is known since the outer boundary is known. By the linearity of the partial differential equation and boundary conditions on $\partial \Omega$ and $\partial D$, we have that the voltage to electrostatic potential mappings 
$$ f \longmapsto u \quad \text{and} \quad f \longmapsto u_0$$
are bounded linear operators from $H^{1/2} (\partial \Omega)$ to $\widetilde{H}^{1} (\Omega)$. Note that $u_0 \in \widetilde{H}^{1} (\Omega)$ by interior elliptic regularity. We now define the \textit{Dirichlet-to-Neumann} (DtN) mappings as 
$$ \Lambda \enspace \text{and} \enspace \Lambda_0 : H^{1/2} ( \partial \Omega) \longrightarrow H^{-1/2} ( \partial \Omega)$$ 
where 
$$ \Lambda f = \partial_{\nu} u \big|_{\partial \Omega} \quad \text{and} \quad \Lambda_{0} f = \partial_{\nu} u_0 \big|_{\partial \Omega}.$$
By appealing to Theorem \ref{gen-soloperator} and the well-posedness of \eqref{hl}, we have that the DtN mappings are bounded linear operators by Trace Theorems. Our main goal is to solve the \textit{inverse shape problem} of recovering the boundary $\partial D$ from the knowledge of the difference of the DtN mapping. That is, we want to determine the boundary $\partial D$ from the difference of all possible measurements $(f, \partial_{\nu} u)$ and $(f, \partial_{\nu} u_0)$. The difference of the normal derivatives $\partial_{\nu}u$ and $\partial_{\nu}u_0$ on the outer boundary $\partial \Omega$ is the current gap imposed on the system by the presence of the defective region $D$. By analyzing the data operator $(\Lambda - \Lambda_0)$, we wish to solve the inverse shape problem by deriving a computationally simple algorithm to detect the defective region(s) via the regularized factorization method. Furthermore, we examine the cases where the interior boundary parameters are real-valued and complex-valued. However, we first show how the DtN mapping $\Lambda$ uniquely determines the boundary coefficients.

\section{Uniqueness of the Inverse Impedance Problem}\label{uniqueness-section}
In this section, we study the uniqueness of the inverse impedance problem of determining the boundary parameters $\mu$ and $\gamma$ provided that the boundary $\partial D$ is given. We refer to \cite{Heejin1,Haddar1} for some results on the uniqueness of the of the inverse impedance problem for other models. We will establish the uniqueness of the boundary parameters $\mu$ and $\gamma$ based on the knowledge of the DtN operator $\Lambda$. To this end, we first consider the following density result.

\begin{theorem}\label{density}
The set
\begin{align*}
\mathscr{U} = \left\{ u \big |_{\partial D} \in H^1(\partial D) \enspace \text{such that} \enspace u \in \widetilde{H}^1(\Omega) \enspace \text{solves} \enspace \eqref{gen-pde} \enspace \text{for any} \enspace f \in H^{1/2}(\partial \Omega) \right\}
\end{align*}
is dense in $H^1(\partial D)$.
\end{theorem}
\begin{proof}
The set $\mathscr{U}$ is a linear subspace of $H^1(\partial D)$ since the mapping $f \mapsto u$ from $H^{1/2}(\partial \Omega)$ to $H^1(\Omega)$ is linear. To show the density of $\mathscr{U}$ in $H^1(\partial D)$, it suffices to show that $\mathscr{U}^\perp$ is trivial. To this end, let $\phi \in \mathscr{U}^\perp$ and $v \in \widetilde{H}^1_0(\Omega)$ be the unique solution to
\begin{align*}
- \Delta v = 0 \enspace \text{in} \enspace \Omega \textbackslash \partial D \quad \text{with}  \quad [\![\partial_\nu v ]\!] \big|_{\partial D} - \overline{\mathcal{B}}(v) =\phi \enspace \text{on} \enspace \partial D,
\end{align*}
where the complex conjugation of $\mathcal{B}$ is given by $\overline{\mathcal{B}}(v) = - \nabla_{\partial D} \cdot \overline{\mu} \nabla_{\partial D} v + \overline{\gamma} v$. Then, for any $f \in H^{1/2}(\partial \Omega)$ we have that the $u$ being the uniques solution to \eqref{gen-pde} satisfies
\begin{align}
0 = \int_{\partial D} u\overline{\phi} \, \text{d}s. \label{denv}
\end{align}
Applying Green's 2nd Theorem in $\Omega \textbackslash \overline{D}$ and $D$, respectively, we have that
$$ 0 = \int_{\partial \Omega} \overline{v} \partial_{\nu} u -u \partial_{\nu} \overline{v} \, \text{d}s  - \int_{\partial D} u [\![\partial_\nu \overline{v} ]\!]  - \overline{v} [\![\partial_\nu u ]\!] \, \text{d}s.  $$
Using the boundary conditions on $\partial D$ and $\partial \Omega$, we have that 
$$ \int_{\partial \Omega} f \partial_\nu \overline{v} \, \text{d}s = -  \int_{\partial D} u \overline{\big(\overline{\mathcal{B}}(v) + \phi \big)}  - \overline{v} {\mathcal{B}}(u)\, \text{d}s.  $$
By appealing to the symmetry of the boundary operator $\mathcal{B}(\cdot)$ we obtain that 
$$\int_{\partial \Omega} f \partial_\nu \overline{v} \, \text{d}s = 0 \quad \text{for any} \enspace f \in H^{1/2}(\partial \Omega)$$
by equation \eqref{denv}. Thus, we have that $\partial_\nu  v \big |_{\partial \Omega} = 0$. Furthermore, since $v \big |_{\partial \Omega}=0$, then by Holmgren's Theorem(see for e.g. \cite{holmgren}), we have that $v = 0$ in $\Omega \textbackslash \overline{D}$. Note that since $\Delta v = 0$ in $D$ with $v^{-} \big |_{\partial D}=0$, then $v$ vanishes in $D$. From the boundary condition, we conclude that $\phi = 0$, which proves the result.
\end{proof}

Now, we can prove that the DtN mapping $\Lambda$ uniquely determines the coefficients $\mu$ and $\gamma$ from the above theorem. Here we assume that $\mu$ is a continuous function on $\partial D$ in addition to our assumptions in Section \ref{dp-ip}. The following theorem is valid for the coefficient parameters $\mu$ and $\gamma$ that can be either real-valued or complex-valued.

\begin{theorem}\label{coeff_uniq}
Assume that $\mu$ and $\gamma$ satisfy the assumptions \eqref{gamma-bounds}--\eqref{mu-bounds}. In addition, assume that $\mu \in C(\partial D)$. Then, the mapping $(\mu, \gamma) \longmapsto \Lambda$ is injective.
\end{theorem}
\begin{proof}
Given $f \in H^{1/2}(\partial \Omega)$, let $u_i$ be the solution to \eqref{gen-pde} with boundary parameters $(\mu_i, \gamma_i)$ and $\Lambda_i$ be the corresponding DtN operator for each $i=1,2$. Assume that the DtN operators $\Lambda_1$ and $\Lambda_2$ coincide. Then, $\partial_\nu u_1  = \partial_\nu u_2$ and $u_1 = u_2$ on $\partial \Omega$, which implies that $u_1 = u_2$ in $\Omega \setminus \overline{D}$ from Holmgren's Theorem. Moreover, due to the fact the $u_i$ is Harmonic in $D$ we have that $u_1=u_2$ in $D$. From the boundary conditions on $\partial D$, we obtain
\begin{align*}
0=  \left( \nabla_{\partial D} \cdot \mu_1\nabla_{\partial D} u_1 - \gamma_1u_1 \right) - \left( \nabla_{\partial D} \cdot \mu_2\nabla_{\partial D} u_2 - \gamma_2u_2 \right).
\end{align*}
Since there are no jumps of the trace of $u_i$ across the boundary $\partial D$, we have that
\begin{align*}
0  =\nabla_{\partial D} \cdot (\mu_1-\mu_2)\nabla_{\partial D} u_1 - (\gamma_1-\gamma_2)u_1.
\end{align*}
For any $\phi \in H^1(\partial D)$, consider a function $\psi$ in the dual space of $H^{-1}(\partial D)$ such that
\begin{align*}
\psi := \nabla_{\partial D} \cdot (\mu_1 - \mu_2) \nabla_{\partial D} \overline{\phi} - (\gamma_1-\gamma_2) \overline{\phi}.
\end{align*}
Then,
\begin{align*}
 \int_{\partial D} \psi u_1 \, \text{d}s
=  \int_{\partial D} \left( \nabla_{\partial D} \cdot (\mu_1-\mu_2)\nabla_{\partial D} u_1 - (\gamma_1-\gamma_2)u_1 \right)\overline{\phi}  \, \text{d}s = 0.
\end{align*}
Therefore, $\psi \in \mathscr{U}^\perp$ and from Theorem \ref{density}, we have $\psi = 0$. That is 
\begin{align}\label{coeff}
\nabla_{\partial D} \cdot (\mu_1 - \mu_2) \nabla_{\partial D} \phi - (\gamma_1-\gamma_2) \phi=0, \quad \forall \phi \in H^1(\partial D).
\end{align}
If we take $\phi = 1$, we have $\gamma_1 = \gamma_2$ a.e. on $\partial D$.

It remains to show that $\mu_1 = \mu_2$. From \eqref{coeff}, we have that for any $\phi \in H^1(\partial D)$, 
\begin{align}\label{eqmu}
0 = \int_{\partial D} \overline{\phi}  \left(\nabla_{\partial D} \cdot (\mu_1 - \mu_2) \nabla_{\partial D} \phi \right) \, \text{d}s
=-\int_{\partial D}  \nabla_{\partial D} \overline{\phi} \cdot (\mu_1-\mu_2)\nabla_{\partial D} \phi \, \text{d}s.
\end{align}
If $\mu_1(x_0) \neq \mu_2(x_0)$ for some $x_0 \in \partial D$, without loss of generality, we have that
\begin{align*}
 - \overline{\xi} \cdot \text{Re}(\mu_1 -\mu_2 ) (x_0) \xi \geq \beta_0 |\xi|^2, \quad \forall \xi \in \mathbb{C}^{d-1}
\end{align*}
for some positive constant $\beta_0$. Since $\mu_1$ and $\mu_2$ are continuous, there exists a ball $B_\varepsilon(x_0)$ of radius $\varepsilon>0$ centered at $x_0$ such that 
\begin{align*}
-  \overline{\xi} \cdot \text{Re}(\mu_1 - \mu_2 ) (x) \xi \geq \beta_0/2 |\xi|^2 \quad \forall x \in \partial D \, \cap \, B_\varepsilon(x_0).
\end{align*}
Consider a smooth function $\phi$ compactly supported in $ \partial D \, \cap \, B_\varepsilon(x_0)$. By taking the real part of \eqref{eqmu}, we have that
\begin{align*}
0 = \int_{\partial D} - \nabla_{\partial D} \overline{\phi} \cdot \text{Re}(\mu_1-\mu_2)\nabla_{\partial D} \phi \, \text{d}s
 \geq \int_{\partial D \, \cap \, B_\epsilon(x_0)} \beta_0/2 | \nabla_{\partial D} \phi |^2,
\end{align*}
which implies that $\nabla_{\partial D} \phi = 0$ on $\partial D \, \cap \, B_\epsilon(x_0)$. Therefore, $\phi$ is a constant on $\partial D \cap B_\epsilon(x_0)$, which is a contradiction. Thus, $\mu_1 = \mu_2$ on $\partial D$.
\end{proof}

In the following sections, we consider two separate cases where the interior boundary coefficients are real-valued and complex-valued. In both instances, we rigorously demonstrate a factorization of an operator derived from $\Lambda$. With these factorizations, we develop range-based identities in order to recover the unknown, extended region $D$.

\section{\textbf{Reconstruction for the Inverse Shape Problem}}\label{invshapeG-EIT}

\subsection{\textbf{Complex--valued boundary coefficients}}\label{complex-section}
In this section, we study the case when the boundary parameters $\mu$ and $\gamma$ are complex-valued. The methodology used here is influenced by the work in \cite{harris2}. The analysis is based on the factorization of the current gap operator $(\Lambda - \Lambda_0)$. The goal is to derive an imaging functional using the singular value decomposition of the known current gap operator.

We begin by defining the auxiliary operator that will be used to derive our sampling method. To this end, for a given $h \in H^{-1}(\partial D)$, we define $w \in \widetilde{H}^{1}_{0}( \Omega)$ to be the unique solution of the adjoint problem to \eqref{gen-pde} given by 
\begin{equation}\label{adj-prob}
- \Delta w = 0 \quad \text{in} \quad \Omega \backslash \partial D \qquad \text{with} \qquad [\![\partial_\nu w ]\!] \big \rvert_{\partial D} = \overline{\mathcal{B}}(w) + h \quad \text{on} \quad \partial D,
\end{equation}
where the overline denotes complex conjugation of the coefficients. Just as in the previous section, one can show that \eqref{adj-prob} is well-posed by appealing to a variational argument. Thus, we can define the bounded linear operator
\begin{equation}\label{adj-opF}
F: H^{-1}(\partial D) \rightarrow H^{-1/2} (\partial \Omega) \quad \text{given by} \quad Fh = \partial_{\nu} w \big \rvert_{\partial \Omega}, 
\end{equation}
where $w$ is the unique solution to \eqref{adj-prob}. We proceed by studying some important properties of the operator $F$ which will be useful in our sampling algorithm.

\begin{theorem}\label{F-comp-inj}
The operator $F$ defined in \eqref{adj-opF} is compact and injective.
\end{theorem}

\begin{proof}
We begin by showing compactness. Since $\text{dist}( \partial \Omega , \overline{D} ) > 0$, there exists a region $\Omega^{*}$ such that $D \subset \Omega^{*} \subset \Omega$ where $\Omega^{*}$ is $\mathcal{C}^2$--smooth. By standard elliptic regularity results(see for e.g. \cite{evans}), we have that the solution $w$ to \eqref{adj-prob} is in $H^{2}(\Omega \backslash \Omega^{*})$ for any $h \in H^{-1}(\partial D)$.  The Trace Theorem implies that $Fh \in H^{1/2}(\partial \Omega)$. Thus, the compact embedding of $H^{1/2}(\partial \Omega)$ into $H^{-1/2}(\partial \Omega)$ proves that $F$ is compact.

To prove injectivity, assume that $h \in Null(F)$. This implies that the solution $w$ to \eqref{adj-prob} with boundary data $h$ on $\partial D$ has zero Cauchy data on $\partial \Omega$. By Holmgren's Theorem, we have that $w=0$ in $\Omega \backslash \overline{D}$. This implies that $w^{+}\big \rvert_{\partial D} =w^{-} \big \rvert_{\partial D} = 0$ and that $\Delta w=0$ in $D$. Therefore, we have that $w = 0$ in $D$ which implies that $[\![\partial_\nu w ]\!] \big \rvert_{\partial D} =\overline{\mathcal{B}}({w})= 0$ on $\partial D$, proving $h = 0$.
\end{proof}

To proceed, we define the sesquilinear dual--product on the closed curve/surface $\Gamma$ as
\begin{equation}\label{out-dualprod}
\langle \varphi , \psi \rangle_{\Gamma} = \int_{\Gamma} \varphi \, \overline{\psi} \, \text{d}s \quad \text{for all} \quad \varphi \in H^{p} ( \Gamma ) \quad \text{and} \quad \psi \in H^{-p} ( \Gamma )
\end{equation}
between the Hilbert Space $H^{p}(\Gamma )$ and its dual space $H^{-p} ( \Gamma )$ for $p>0$ where $L^2 ( \Gamma)$ is the Hilbert pivot space. Recall, that we have the following
$$ H^{1} ( \Gamma ) \subset H^{1/2} ( \Gamma) \subset L^2( \Gamma) \subset H^{-1/2} ( \Gamma ) \subset H^{-1} ( \Gamma )$$ 
with dense inclusions. In this paper, we are particularly interested in the cases when $\Gamma = \partial D$ and $\Gamma = \partial \Omega$. These dual-products will also be used in the upcoming sections. In the analysis of this section, we will need the adjoint operator of $F$ with respect to the sesquilinear forms $\langle \cdot , \cdot \rangle_{\partial \Omega} $ and $ \langle \cdot , \cdot \rangle_{\partial D}$ which is given by the following theorem.

\begin{theorem}\label{f-adj}
The adjoint operator 
$$F^{*} : H^{1/2}(\partial \Omega) \rightarrow H^{1}(\partial D) \quad \text{is given by} \quad F^{*}f = u \big \rvert_{\partial D}. $$
Moreover, $F^{*}$ is injective, i.e. $F$ has dense range.
\end{theorem}

\begin{proof}
We begin by applying Green's second theorem to the solution $u$ of \eqref{gen-pde} and the solution $w$ to the adjoint problem \eqref{adj-prob} on the regions $\Omega \backslash \overline{D}$ and $D$ in order to obtain
$$ 0 = \int_{\partial \Omega}  \overline{w} \partial_{\nu} u - u \partial_{\nu} \overline{w} \, \text{d}s  - \int_{\partial D} u [\![\partial_\nu \overline{w} ]\!]  - \overline{w} [\![\partial_\nu u ]\!] \, \text{d}s.  $$
By applying the boundary conditions on $\partial \Omega$ and $\partial D$, we obtain
$$\int_{\partial \Omega} f \partial_{\nu} \overline{w} \, \text{d}s = \int_{\partial D} u \big(\overline{h + \overline{\mathcal{B}}(w)}\big)  - \overline{w} \mathcal{B}(u) \, \text{d}s = \int_{\partial D} u \overline{h} \, \text{d}s .$$
The above equality implies that $F^{*}f = u \big \rvert_{\partial D}$ since \eqref{out-dualprod} implies that
$$\int_{\partial \Omega} f \partial_{\nu} \overline{w} \, \text{d}s= \langle f , Fh \rangle_{\partial \Omega} = \langle F^{*}f , h \rangle_{\partial D} = \int_{\partial D} u \overline{h} \, \text{d}s $$
which proves the first part of our result.

To show that $F^{*}$ is injective, suppose that $f \in Null(F^{*})$. Then we have that $\mathcal{B}(u) = 0$, which implies that $u$ is the unique solution to the Dirichlet problem on $D$ with zero Dirichlet data. Thus, $u = 0$ in $D$. Furthermore, the generalized Robin boundary condition 
$$[\![\partial_\nu u ]\!] \big \rvert_{\partial D} = \mathcal{B}(u) \,\, \text{on $\partial D$} \quad  \text{implies that} \quad  \partial_{\nu} u  \big \rvert^{+}_{\partial D} = u \big \rvert^{+}_{\partial D} = 0.$$
 Note that $u$ is harmonic on $\Omega \backslash \overline{D}$ with zero Cauchy data on $\partial D$. Using Holmgren's Theorem and the Trace Theorem, we have that $f = 0$. Thus, $F^{*}$ is injective, which implies that $F$ has dense range(see for e.g. \cite{mclean-book}).
\end{proof}

Sampling methods typically connect the region of interest to an ill-posed equation involving the data operator. In the two cases we are considering, we will use a singular solution to the background problem, i.e. the equation where the region of the interest is not present. Using the singularity of the solution to the background problem, one can show that an associated ill-posed problem is solvable if and only if the singularity is contained in the region of interest. To this end, we define the Dirichlet Green's function for the negative Laplacian for the known domain $\Omega$ as $\mathbb{G}(\cdot , z) \in H^{1}_{loc}(\Omega \backslash \brac{z})$, for $z \in \Omega$, be the unique solution to the boundary value problem
$$- \Delta \mathbb{G}(\cdot , z) = \delta (\cdot , z) \quad \text{in} \quad \Omega \qquad \text{and} \qquad \mathbb{G}(\cdot ,z) \big \rvert_{\partial \Omega} = 0 .$$
The following result shows that $Range(F)$ uniquely determines the region $D$.

\begin{theorem}\label{adj-complex-case}
The operator $F$ defined in \eqref{adj-opF} is such that 
$$\partial_{\nu} \mathbb{G}(\cdot , z) \in Range(F) \quad \text{if and only if} \quad z \in D.$$
\end{theorem}

\begin{proof}
To prove the claim, we first assume that $z \in \Omega \backslash \overline{D}$. Suppose by contradiction that $\partial_{\nu} \mathbb{G}(\cdot , z) \in Range (F)$, i.e. there exists $h_z \in H^{-1}(\partial D)$ such that $F h_z = \partial_{\nu} \mathbb{G}(\cdot , z) \big \rvert_{\Omega}$. This implies that there exists $v_z \in \widetilde{H}^{1}_{0}(\Omega)$ such that
$$-\Delta v_z = 0 \enspace \text{in} \enspace \Omega \backslash \partial D \quad \text{with} \quad [\![ \partial_{\nu} v_z ]\!] \big \rvert_{\partial D} = \overline{\mathcal{B}}({v}_z) + h_z  \,\, \text{on $\partial D$}.$$
Furthermore, $\partial_{\nu}v_z \big \rvert_{\partial \Omega} = \partial_{\nu} \mathbb{G}(\cdot , z) \big \rvert_{\partial \Omega}$ and we have that $v_z$ satisfies
$$-\Delta v_z = 0 \enspace \text{in} \enspace \Omega \backslash \partial D \quad \text{with} \quad v_z \big \rvert_{\partial \Omega} = 0 \quad \text{and} \quad \partial_{\nu}v_z \big \rvert_{\partial \Omega} = \partial_{\nu} \mathbb{G}(\cdot , z) \big \rvert_{\partial \Omega}. $$
So we define $V_z = v_z - \mathbb{G}(\cdot , z)$ and note that
$$- \Delta V_z = 0 \enspace \text{in} \enspace \Omega \backslash (\overline{D} \cup \brac{z}) \quad \text{with} \quad V_z \big \rvert_{\partial \Omega} = 0 \quad \text{and} \quad \partial_{\nu} V_z \big \rvert_{\partial \Omega} = 0.$$
By Holmgren's Theorem \cite{holmgren}, we can conclude that $v_z = \mathbb{G}(\cdot , z)$ in $\Omega \backslash (\overline{D} \cup \brac{z})$. By interior elliptic regularity, $v_z$ is continuous at $z \in \Omega \backslash \overline{D}$. However, $\mathbb{G}(\cdot ,z)$ has a singularity at $z$, which proves the claim by contradiction due to the fact that
$$|v_z (x)| < \infty \quad \text{whereas} \quad |\mathbb{G}(\cdot ,z)| \rightarrow \infty \enspace \text{as} \enspace x \rightarrow z.$$
Conversely, we now assume that $z \in D$ and let $\eta \in H^{1}(D)$ be the solution to the following Dirichlet problem in $D$
$$- \Delta \eta = 0 \enspace \text{in} \enspace D \quad \text{with} \quad \eta \big \rvert_{\partial D}^{-} = \mathbb{G}(\cdot ,z)\big \rvert_{\partial D}^{+}. $$
Now, define $v_z$ such that 
\[ v_z = \begin{cases} 
          \mathbb{G}( \cdot \, , z) & \text{in} \enspace \Omega \textbackslash \overline{D}\\
          \eta & \text{in} \enspace D
       \end{cases}
    \]
and we will show that $v_z$ satisfies \eqref{adj-prob} for some $h_z \in H^{-1}(\partial D)$. By definition, we have that $v_z$ is harmonic  in $\Omega \backslash \partial D$. Note that since $z \in D$, then $\mathbb{G}(\cdot , z) \in H^2(\Omega \backslash \overline{D})$ which implies that $v_z \in \tilde{H}^{1}_{0} (\Omega)$. By construction, we have that $\partial_{\nu} v_z \big \rvert_{\partial \Omega} = \partial_{\nu} \mathbb{G}(\cdot , z) \big \rvert_{\partial \Omega}$. Now, we need to prove that
$$h_z = [\![\partial_\nu v_z ]\!] \big \rvert_{\partial D} - \overline{\mathcal{B}}({v}_z) \quad \text{on $\partial D$}$$
is in $H^{-1}(\partial D)$. Notice that 
$$ [\![\partial_\nu v_z ]\!] \big \rvert_{\partial D} = \partial_{\nu} \mathbb{G}(\cdot , z) \big \rvert_{\partial D}^{+} - \partial_{\nu} \eta \big \rvert_{\partial D}^{-} .$$
Therefore, we have that $\mathbb{G}(\cdot , z) \big \rvert_{\partial D}^{+} \in H^{3/2}(\partial D)$ which implies that $\eta \in H^{2}(D)$ by appealing to elliptic regularity. By the Neumann Trace Theorem, we obtain that 
$$ [\![\partial_\nu v_z ]\!] \big \rvert_{\partial D} \in H^{1/2}(\partial D) \subset H^{-1}(\partial D) .$$
Also, it is clear that $\overline{\mathcal{B}}({v}_z) \in H^{-1}(\partial D)$ since the trace of $v_z$ on $\partial D$ is in $H^1(\partial D)$. Thus, we can conclude that $h_z \in H^{-1}(\partial D)$. By the definition of the operator $F$, we have that $F h_z = \partial_{\nu} \mathbb{G}(\cdot , z) \big \rvert_{\partial \Omega}$, proving the claim.
\end{proof}

We have shown that the operator $F$ uniquely determines the region of interest $D$. Our next task is to connect the range of $F$ to the range of an operator derived from $(\Lambda - \Lambda_0)$. The following result will provide important properties of the Direchlet-to-Neumann mapping which will be used in our sampling method.

\begin{theorem}\label{cgo-comp-id}
The current gap operator defined by $(\Lambda - \Lambda_0)f = \partial_{\nu}(u - u_0) \big \rvert_{\partial \Omega}$ where $u$ and $u_0$ are solutions to \eqref{gen-pde} and \eqref{hl}, respectively, is compact. Moreover, we have the identity
$$\langle f, (\Lambda - \Lambda_0)f \rangle_{\partial \Omega} = \int_{\Omega} |\nabla u|^2 \, \text{d}x + \int_{\partial D} \overline{\mu} |\nabla_{\partial D} u|^2 + \overline{\gamma} |u|^2 \, \text{d}s - \int_{\Omega} |\nabla u_0 |^2 \, \text{d}x $$ 
\end{theorem}

\begin{proof}
To show compactness, we follow a similar procedure from the proof of \thmref{F-comp-inj} and is omitted to avoid repetition. 

To prove the identity, we have that by definition
$$ \langle f, (\Lambda - \Lambda_0)f \rangle_{\partial \Omega} = \int_{\partial \Omega} f \partial_{\nu} \overline{u} \, \text{d}s - \int_{\partial \Omega} f \partial_{\nu} \overline{u_0} \, \text{d}s = \int_{\partial \Omega} u \partial_{\nu} \overline{u} \, \text{d}s - \int_{\partial \Omega} u_0 \partial_{\nu} \overline{u_0} \, \text{d}s$$
By Green's first identity on the regions $\Omega \backslash \overline{D}$ and $D$, we have that 
$$\langle f, (\Lambda - \Lambda_0)f \rangle_{\partial \Omega} = \int_{\Omega} |\nabla u|^2 \, \text{d}x + \int_{\partial D} u [\![\partial_\nu \overline{u} ]\!] \, \text{d}s - \int_{\Omega} |\nabla u_0 |^2 \, \text{d}x .$$
From the general Robin boundary condition on $\partial D$, we obtain that
$$\langle f, (\Lambda - \Lambda_0)f \rangle_{\partial \Omega} = \int_{\Omega} |\nabla u|^2 \, \text{d}x + \int_{\partial D} \overline{\mu} |\nabla_{\partial D} u|^2 + \overline{\gamma} |u|^2 \, \text{d}s - \int_{\Omega} |\nabla u_0 |^2 \, \text{d}x $$ 
which proves the claim.
\end{proof}
In order to prove the main result of this section, we define the imaginary part of the current gap operator as
$$\text{Im}(\Lambda - \Lambda_0) = \dfrac{1}{2 \text{i}} \big[ (\Lambda - \Lambda_0) - (\Lambda - \Lambda_0)^{*} \big ] . $$
By \thmref{cgo-comp-id}, we have that
$$\text{Im} \langle f, (\Lambda - \Lambda_0)f \rangle_{\partial \Omega} = \int_{\partial D} \text{Im}(\overline{\mu}) |\nabla_{\partial D} u|^2 + \text{Im}(\overline{\gamma}) |u|^2 \, \text{d}s $$
Recall, that our assumption that the boundary parameters satisfy $-\overline{\xi} \cdot \text{Im}(\mu) \xi  \geq \zeta_2 |\xi|^2> 0$ for all $\xi \in \C^{d-1} \setminus \{0\}$ as well as $-\text{Im}(\gamma) \geq \beta_2 > 0$ a.e. on $\partial D$. Thus, we have that there are constants $C_1 , C_2 > 0$ such that
$$ C_1 \norm{F^{*} f}_{H^{1}(\partial D)}^2 \leq \text{Im} \langle f, (\Lambda - \Lambda_0)f \rangle_{\partial \Omega} \leq C_2 \norm{F^{*} f}_{H^{1}(\partial D)}^2 .$$
The compactness of $\text{Im}(\Lambda - \Lambda_0)$ and injectivity of $F^{*}$ further implies that $\text{Im}(\Lambda - \Lambda_0)$ is a positive compact operator. Thus, there exists a compact operator $Q: H^{1/2}(\partial \Omega) \rightarrow L^{2}(\partial \Omega)$ such that
the imaginary-part of the data operator has the following symmetric factorization
$$\text{Im}(\Lambda - \Lambda_0) = Q^{*}Q .$$
Therefore, we have that 
$$  C_1 \norm{F^{*} f}_{H^{1}(\partial D)}^2 \leq \norm{ Q^{*}f }_{H^{-1/2}(\partial \Omega )}^{2} \leq C_2 \norm{F^{*} f}_{H^{1}(\partial D)}^2$$
for all $f \in H^{1/2}(\partial \Omega)$. In order to finish proving the main result of this section, we state an important lemma connecting the ranges of $Q$ and $F$ which is required by our sampling method. For the proof of the following result we refer to \cite{embry} and \cite{eit-harrach1} where the arguments for real Hilbert spaces can be generalized to Banach spaces.
\begin{lemma}\label{banach-lemma}
Let $A_j$ be bounded linear operators mapping $X_i \rightarrow Y$ where $X_i$ and $Y$ are Banach spaces for $i=1,2$. If
$$\exists \, c_1 , c_2 > 0 \quad \text{such that} \quad c_1 \norm{A_{1}^{*} f}_{X_{1}^{*}} \leq \norm{ A_{2}^{*}f }_{X_{2}^{*}} \leq c_2 \norm{A_{1}^{*} f}_{X_{1}^{*}} $$
for all $f \in Y^{*}$, then $Range(A_1) = Range(A_2)$.
\end{lemma}
By the above inequalities and \lemref{banach-lemma}, we have the following result.
\begin{theorem}\label{ranges}
If the boundary coefficients $\gamma$ and $\mu$ satisfy \eqref{gamma-bounds} and \eqref{mu-bounds}, respectively, then
$$Range(F) = Range(Q) .$$
\end{theorem}
This allows one to uniquely recover the defective region $D$ from the knowledge the DtN mapping $\Lambda$. Recall, that $Q$ is determined from the imaginary-part of the measured current-gap operator. By all of the theorems of this section and the results of Theorem 2.3 of \cite{harris1}, we create an explicit characterization that will allow us to detect the delaminated region. In our case we have, 
$$ \ell \in Range(Q) \quad \text{if and only if} \quad \liminf_{\alpha \rightarrow 0} \langle f_{\alpha} , \text{Im}(\Lambda - \Lambda_0 ) f_{\alpha} \rangle_{\partial \Omega} < \infty$$ 
where $f_{\alpha}$ is the regularized solution to $\text{Im}(\Lambda - \Lambda_0 ) f = \ell$. By appealing to Theorems \ref{adj-complex-case} and Lemma \ref{banach-lemma} we have that 
\begin{align}\label{complex-singularity}
\partial_{\nu} \mathbb{G}(\cdot , z) \in Range(Q) \quad \text{if and only if} \quad z \in D.
\end{align}
With this, equation \eqref{complex-singularity} and the results of \cite{harris1}, we are able to finally provide the main result of this section. 

\begin{theorem}\label{complex-case-thm}
The  imaginary part of the current-gap operator $\text{Im}(\Lambda - \Lambda_0) : H^{1/2}(\partial \Omega) \rightarrow H^{-1/2}(\partial \Omega)$ uniquely determines $D$ such that for any $z \in \Omega$
$$z \in D \quad \text{if and only if} \quad \liminf_{\alpha \rightarrow 0} \langle f_{\alpha}^{z} , \text{Im}(\Lambda - \Lambda_{0}) f_{\alpha}^{z} \rangle_{\partial \Omega} < \infty$$ 
where $f_{\alpha}^{z}$ is the regularized solution to $\text{Im}(\Lambda - \Lambda_0) f^{z} = \partial_{\nu} \mathbb{G}(\cdot , z ) \big \rvert_{\partial \Omega}$. \end{theorem}

Note that a regularized solution is required since $\text{Im}(\Lambda - \Lambda_0)$ is compact. However, since the operator is injective with dense range, we may utilize any regularization technique such as Tikhonov or Spectral cut-off. With our main result, we are able to successfully characterize every point in the known domain $\Omega$ as either inside or outside the region of interest $D$ for the case where the boundary coefficients $\mu$ and $\gamma$ are complex-valued. Furthermore, we show that the DtN mapping $\Lambda$ uniquely determines the damaged region $D$. In other words, one is able to reconstruct the region $D$ from physical measurements on the accessible boundary $\partial \Omega$.

\subsection{Real--valued boundary coefficients}\label{real-section}
In this section, we study the case when the interface parameters $\mu$ and $\gamma$ are strictly real-valued. Note that the well-posedness argument in this case is identical to the one provided in Section \ref{dp-ip} since we are synonymously assuming $\text{Im}(\gamma )=0$ and $\text{Im}(\mu) = 0$. We will derive a symmetric factorization for the current-gap operator $(\Lambda - \Lambda_0)$ as similarly done in Section \ref{complex-section}, where the theory was developed in \cite{harris1}. Thus, we will provide another algorithm for recovering the unknown region $D$ from the measurements operator given by the current gap operator $(\Lambda - \Lambda_0)$ for real-valued parameters.

Inspired by the current gap operator $(\Lambda - \Lambda_0)$, we note that $(u - u_0) \in \widetilde{H}^{1}_{0}(\Omega)$ solves
$$- \Delta (u - u_0) = 0 \quad \text{in} \quad \Omega \backslash \partial D \qquad \text{with} \qquad [\![\partial_\nu (u - u_0) ]\!] \big|_{\partial D} = \mathcal{B}(u) \quad \text{on $\partial D$.} $$
So we define $w \in \widetilde{H}^{1}_{0}(\Omega)$ to be the unique solution to the auxiliary problem
\begin{equation}\label{gen-aux-prob}
- \Delta w = 0 \quad \text{in} \quad \Omega \backslash \partial D \qquad \text{with} \qquad [\![\partial_\nu w ]\!] \big|_{\partial D} = \mathcal{B}(h)  \quad \text{on $\partial D$.} 
\end{equation}
for any given $h \in H^{1}(\partial D)$. One can show that \eqref{gen-aux-prob} is well-posed by appealing to a variational formulation argument as in Section \ref{dp-ip}. Thus, we can define the bounded linear Source-to-Neumann operator
$$G:H^{1}(\partial D) \rightarrow H^{-1/2}(\partial \Omega) \quad \text{given by} \quad Gh = \partial_{\nu} w \big \rvert_{\partial \Omega}$$
where $w$ is the unique solution to \eqref{gen-aux-prob}. In order to understand the connection between the operators $G$ and $(\Lambda - \Lambda_0)$, note that by the well-posedness of \eqref{gen-aux-prob} we have that
$$\partial_{\nu} w\big \rvert_{\partial \Omega}  = (\Lambda - \Lambda_0)f \quad \text{provided that} \quad h = u \big \rvert_{\partial D}.$$
From this, we define the solution operator for the electrostatic potential $u$ such that
$$S: H^{1/2}(\partial \Omega) \rightarrow H^{1}(\partial D)  \quad \text{given by} \quad Sf = u \big \rvert_{\partial D}. $$
Therefore, we obtain the initial factorization $(\Lambda - \Lambda_0)f = GSf$ for any $f \in H^{1/2}(\partial \Omega)$. In order to further factorize the operator $(\Lambda - \Lambda_0)$, we need to decompose $G$. In order to do so, we will compute and analyze the adjoint of the solution operator $S$. The adjoint operator $S^{*}$ is detailed in the following result.
\begin{theorem}\label{gen-adjoint}
The adjoint operator $S^{*}: H^{-1} ( \partial D) \rightarrow H^{-1/2} ( \partial \Omega )$ is given by $S^{*}g = \partial_{\nu} v \big|_{\partial \Omega}$ where $v \in \widetilde{H}^1_0 (\Omega)$ satisfies 
\begin{equation} \label{gen-adv}
- \Delta v = 0 \quad \text{in} \quad  \Omega \textbackslash \partial D \quad \text{with} \quad [\![\partial_\nu v ]\!] \big|_{\partial D} = \mathcal{B}(v)  + g \,\,\, \text{on $\partial D$.}
\end{equation} 
Moreover, the operator $S$ is injective.
\end{theorem}
\begin{proof} 
Notice, that by using a variational argument we can establish that the solution $v \in \widetilde{H}_{0}^{1} ( \Omega )$ exists, is unique, and continuously depends on $g \in H^{-1} ( \partial D)$. Using a similar technique used in the proof of Theorem \ref{f-adj}, we have that 
$$ 0 =  \int_{\partial \Omega} \overline{v} \, \partial_{\nu}u - u \, \partial_{\nu} \overline{v} \, \text{d}s - \int_{\partial D} \overline{v}  [\![\partial_\nu {u} ]\!]  \, \text{d}s + \int_{\partial D} u  [\![\partial_\nu \overline{v} ]\!]  \, \text{d}s. $$
Thus, by the boundary conditions on $\partial \Omega$ we have that
$$  \int_{\partial \Omega}  f \, \partial_{\nu} \overline{v} \, \text{d}s =  \int_{\partial D} u [\![\partial_\nu \overline{v} ]\!] \, \text{d}s -\int_{\partial D} \overline{v} [\![\partial_\nu u ]\!]\, \text{d}s. $$
By the boundary condition on $\partial D$ for $u$ and $v$, we have that
\begin{align*}
\int_{\partial D} u [\![\partial_\nu \overline{v} ]\!] \, \text{d}s  -  \int_{\partial D} \overline{v} [\![\partial_\nu u ]\!]\, \text{d}s &=  \int_{\partial D} u \overline{(g+\mathcal{B}(v) )} \, \text{d}s -\int_{\partial D} \overline{v} \mathcal{B}(u)\, \text{d}s \\
	&= \int_{\partial D} u \overline{g} \, \text{d}s
\end{align*} 
With the dual-product on the boundaries $\partial D$ and $\partial \Omega$ as defined in \eqref{out-dualprod}, we have that $$ \langle Sf , g \rangle_{\partial D} = \int_{\partial D} u \overline{g} \, \text{d}s = \int_{\partial \Omega} f \partial_{\nu} \overline{v} \, \text{d}s = \langle f , S^{*} g \rangle_{\partial \Omega}$$
for all $f \in H^{1/2} (\partial \Omega)$ and $g \in H^{-1} (\partial D)$ which implies that $S^{*} g = \partial_{\nu} v \big|_{\partial D}$.

To prove injectivity, we let $Sf = 0$ which implies that $u =0$ in $\bar{D}$. By our boundary condition, we have that $[\![\partial_\nu u ]\!] \big \rvert_{\partial D}= \mathcal{B}(u) = 0$ on $\partial D$. Thus, $\partial_{\nu} u \big \rvert_{\partial D}^{+} = 0$. Using Holmgren's Theorem, we have that $u = 0$ in $\Omega$. Then by the Trace Theorem, we have that $f = 0$ on $\partial \Omega$, proving that $S$ is injective.
\end{proof}

In order to complete the factorization of the current gap operator, we need to define a middle operator $T$. Recall, that $w$ is the unique solution to equation \eqref{gen-aux-prob}, which implies that $w$ is harmonic in $\Omega \textbackslash \partial D$ and
$$ [\![\partial_\nu w ]\!] \big|_{\partial D} = \mathcal{B}(w) + \mathcal{B}(h-w)  \quad \text{on $\partial D$.} $$
Therefore, we have that 
$$\partial_{\nu} w \big|_{\partial \Omega} = Gh \quad \text{as well as} \quad  \partial_{\nu} w \big|_{\partial \Omega} =S^{*}  \mathcal{B}(h-w)$$
by the well-posedness of \eqref{gen-adv} and Theorem \ref{gen-adjoint}. Motivated by this, we define the operator 
$$T: H^{1} ( \partial D ) \rightarrow H^{-1} ( \partial D ) \quad \text{given by} \quad Th = \mathcal{B}(h-w)$$
By the well-posedness of \eqref{gen-aux-prob}, $T$ is a bounded linear operator. Recall, that we had already established that $(\Lambda - \Lambda_0) = GS$ and observe that we have factorized the operator $G$ such that $G = S^{*}T$. This gives the following result.

\begin{theorem}\label{factorization} 
The difference of the DtN mappings $(\Lambda - \Lambda_0): H^{1/2} ( \partial \Omega) \rightarrow H^{-1/2} ( \partial \Omega )$ has the symmetric  factorization $( \Lambda - \Lambda_0 ) = S^{*} TS$. 
\end{theorem}

In order to apply Theorem 2.3 from \cite{harris1} to solve the inverse problem of recovering $D$ from the current gap operator $(\Lambda - \Lambda_0)$, we need to prove that $T$ is coercive as well as characterize the region $D$ by the range of $S^{*}$. The following two results will allow us to prove some useful properties of the current gap operator using the symmetric factorization from the previous theorem. We now prove the coercivity of the operator $T$.

\begin{theorem}\label{gen-tcoercive}
The operator $T: H^{1} ( \partial D ) \rightarrow H^{-1} ( \partial D )$ defined by 
\begin{equation}\label{T-mapping}
Th =\mathcal{B}(h-w)
\end{equation}
is coercive on $H^{-1} ( \partial D )$, where $h \in H^{1}( \partial D)$ and $w \in \widetilde{H}^{1}_{0} (\Omega)$ satisfies \eqref{gen-aux-prob}.
\end{theorem}
\begin{proof} Using the generalized Robin transmission condition on $\partial D$ in equation \eqref{gen-aux-prob}, we have that
\begin{align*}
	\langle h, Th \rangle_{ \partial D }  &= \int_{\partial D} h   \overline{\mathcal{B}(h-w)}  \, \text{d}s \\
	&= \int_{\partial D} \mathcal{B}( h)  \overline{(h-w)} \, \text{d}s \quad \text{ by the symmetry of $\mathcal{B}(\cdot)$} \\
	&= \int_{\partial D} \overline{h} {\mathcal{B}(h)} \, \text{d}s - \int_{\partial D}  \overline{w} [\![\partial_\nu w ]\!]  \, \text{d}s \quad \text{ by equation \ref{gen-aux-prob}} \\
	&= \int_{\partial D} \mu \rvert \nabla_{\partial D} h  \rvert^2 + \gamma |h|^2 \, \text{d}s - \int_{\partial D}  \overline{w} [\![\partial_\nu w ]\!]  \, \text{d}s .
\end{align*} 
Following a similar technique used to derive \eqref{gen-vf}, we have that 
$$ \int_{\Omega \backslash D} |\nabla w|^2 \, \text{d}x = - \int_{\partial D} \overline{w} \partial_{\nu} w^{+}  \, \text{d}s \quad \text{and} \quad \int_{D} |\nabla w|^2 \, \text{d}x = \int_{\partial D} \overline{w} \partial_{\nu} w^{-} \, \text{d}s.$$
Adding both equations above and using the boundary condition on $\partial D$  yields
 $$ \int_{\Omega} |\nabla w|^2 \, \text{d}x = - \int_{\partial D} \overline{w} [\![\partial_\nu w ]\!]  \, \text{d}s.$$ 
This yields the following
\begin{align*}
	\langle h, Th \rangle_{ \partial D }  &= \int_{\partial D}  \mu \rvert \nabla_{\partial D} h  \rvert^2 + \gamma |h|^2 \, \text{d}s + \int_{\Omega} |\nabla w |^{2} \, \text{d}x \\
	&\geq \text{min} \brac{\zeta_1,\beta_1} \int_{\partial D} |\nabla_{\partial D} h|^2 + |h|^2 \, \text{d}s
\end{align*} 
which proves the claim.
\end{proof}

The following two results are critical in allowing us to prove the main theorem of this section which characterizes the analytical properties of the current gap operator.
 
\begin{theorem}\label{gen-DtNprop}
The difference of the DtN mappings $(\Lambda - \Lambda_0 ): H^{1/2} ( \partial \Omega ) \rightarrow H^{-1/2} ( \partial \Omega )$ is compact, injective, and has dense range.
\end{theorem}

\begin{proof} 
To prove compactness follows from Theorem \thmref{cgo-comp-id}. 

We prove that the current gap operator $(\Lambda - \Lambda_0)$ is injective and has dense range using a similar  argument as in \cite{EIT-granados1}. That is, we show that the set of annihilators for $Range(\Lambda - \Lambda_0)$ and $Null(\Lambda - \Lambda_0 )$ are trivial. To this end, note that for all $f,g \in H^{1/2} ( \partial \Omega)$
\begin{align*}
 \langle g , (\Lambda - \Lambda_0) f \rangle_{\partial \Omega} &=  \int_{\partial \Omega} g \, \partial_{\nu} \overline{u (\cdot \, , f)} - g \, \partial_{\nu} \overline{u_{0} (\cdot \,, f)} \, \text{d}s\\
	&= \int_{\partial \Omega} u (\cdot \,, g) \, \partial_{\nu} \overline{u (\cdot \,, f) } - u_{0}(\cdot \,, g) \, \partial_{\nu} \overline{u_{0}(\cdot \, , f)} \, \text{d}s 
\end{align*} 
where the pairs $(u(\cdot , f) , u(\cdot , g))$ and $(u_{0}(\cdot , f) , u_{0}(\cdot , g))$ are solutions to \eqref{gen-pde} and \eqref{hl} with Dirichlet boundary conditions $f$ and $g$ in $H^{1/2}(\partial \Omega)$, respectively. With this, we can use Green's 1st Theorem which implies that 
\begin{align*}
\langle g , (\Lambda - \Lambda_0) f \rangle_{\partial \Omega} &= \int_{\Omega} \nabla u(\cdot \, , g) \cdot \nabla \overline{u(\cdot \, ,f)} \, \text{d}x -  \int_{\Omega} \nabla u_{0}(\cdot \, , g) \cdot \nabla \overline{u_{0}(\cdot \, , f)} \, \text{d}x \\
& \hspace{1in} +  \int_{\partial D} \mathcal{B} \left(u(\cdot \, ,g) \right)\, \overline{u(\cdot \, , f)} \, \text{d}s
\end{align*}  
by the boundary value problems \eqref{gen-pde} and \eqref{hl}. To prove the claim, suppose $f \in H^{1/2} ( \Omega)$ is an annihilator for $Range(\Lambda - \Lambda_0)$ or that $f \in Null (\Lambda - \Lambda_0 )$. In either case, we have that
\begin{align*}
	0   &=  \langle f , (\Lambda - \Lambda_0 ) f \rangle_{\partial \Omega} \\
	&= \int_{\Omega} | \nabla u(\cdot \, , f) |^{2} \, \text{d}x - \int_{\Omega} | \nabla u_{0}(\cdot \, , f) |^{2} \, \text{d}x + \int_{\partial D} \mu |\nabla_{\partial D} u(\cdot , f)|^{2} + \gamma | u(\cdot \, , f) |^{2} \, \text{d}x \\
	&\geq \int_{\partial D} \mu |\nabla_{\partial D} u(\cdot , f)|^{2} + \gamma | u(\cdot \, , f) |^{2} \, \text{d}x
\end{align*} 
where we have used that the harmonic function $u_{0}(\cdot , f)$ minimizes the Dirichlet energy. By \thmref{gen-adjoint}, $S$ is injective which implies that $f=0$, proving both claims.
\end{proof}

All of the theorems of this section imply that the current gap operator $(\Lambda - \Lambda_0)$ satisfies all of the conditions of Theorem 2.3 of \cite{harris1}. Similarly as in Section \ref{complex-section}, we have that $$ \ell \in Range(S^{*}) \quad \text{if and only if} \quad \liminf_{\alpha \rightarrow 0} \langle f_{\alpha} , (\Lambda - \Lambda_0 ) f_{\alpha} \rangle_{\partial \Omega} < \infty$$ where $f_{\alpha}$ is the regularized solution to $(\Lambda - \Lambda_0 ) f = \ell$. Since $(\Lambda - \Lambda_0)$ is compact and injective with a dense range, we can apply any regularization scheme. In a similar way, we show the connection between the domain $D$ and the range of the operator $S^{*}$. We once again use the Dirichlet Green's function for the negative Laplacian for the known domain $\Omega$, $\mathbb{G}(\cdot , z) \in H_{loc}^{1} (\Omega \textbackslash \brac{z})$ for any fixed $z \in \Omega$. Recall that the idea of the following result is to show that due to the singularity at $z$, the normal derivative of the Green's function is not contained in the range of $S^{*}$ unless the singularity is contained within the region of interest $D$.

\begin{theorem} \label{greenchar}
The operator $S^{*}$ is such that for any $z \in \Omega$ 
$$\partial_{\nu} \mathbb{G} (\cdot , z) \big|_{\partial D} \in Range(S^{*}) \quad \text{if and only if} \quad z \in D.$$
\end{theorem}

\begin{proof} 
Notice, since the coefficients $\mu$ and $\gamma$ are real valued, we have that equation \eqref{adj-prob} and the equation in Theorem \ref{gen-adjoint} to define $S^*$ are the same. This implies that the operator $F$ defined by \eqref{adj-opF} and $S^*$ given by Theorem \ref{gen-adjoint} coincide. Therefore, we have that 
$$Range(F) = Range(S^*) $$ 
which gives the result by appealing to Theorem \ref{adj-complex-case}. 
\end{proof}

With \thmref{greenchar}, we have all we need to conclude that the regularized factorization method can be used to recover an unknown region $D$ from the knowledge of the difference of the DtN mappings $(\Lambda - \Lambda_0)$.

\begin{theorem}\label{real-case-thm}
The difference of the DtN mappings $(\Lambda - \Lambda_0) : H^{1/2} ( \partial \Omega) \rightarrow H^{-1/2} (\partial \Omega)$ uniquely determines $D$ such that for any $z \in \Omega$ 
$$z \in D \quad \text{if and only if} \quad \liminf_{\alpha \rightarrow 0 } \langle f_{\alpha}^{z} , (\Lambda - \Lambda_{0}) f_{\alpha}^{z} \rangle_{\partial \Omega} < \infty$$ 
where $f_{\alpha}^{z}$ is the regularized solution to $(\Lambda - \Lambda_0) f^{z} = \partial_{\nu} \mathbb{G}(\cdot , z ) \big \rvert_{\partial \Omega}$. 
\end{theorem}

This concludes the shape reconstruction problem for an extended region for the case when the boundary coefficients are strictly real-valued. In the following section, we provide some numerical experiments for reconstructing $D$.

\section{\textbf{Numerical Validation}}\label{numerical-validation}
In this section, we present numerical examples for the regularized factorization method developed in Sections \ref{complex-section} and \ref{real-section} for solving the inverse shape problem. Our numerical experiments are done in \texttt{MATLAB} 2020a. For simplicity, we will consider the problem in $\mathbb{R}^2$ where $\Omega$ is the unit disk. Notice that the trace spaces $H^{\pm 1/2} (\partial \Omega)$ can be identified with $H_{\text{per}}^{\pm 1/2} [0 , 2 \pi ]$.
 To apply \thmref{complex-case-thm} and \thmref{real-case-thm}, we need the normal derivative of Green's function $\mathbb{G}(\cdot , z)$ with zero Dirichlet condition on the boundary of the unit disk. In polar coordinates, it is well known that the normal derivative of Green's function for the unit disk is given by the Poisson kernel
$$ \partial_{\nu(z)} \mathbb{G}\big ( \cdot \, , z \big ) \big|_{\partial \Omega} = - \frac{1}{2 \pi} \bigg[ \frac{1 - |z|^2 }{|z|^2 + 1 - 2 |z| \text{cos}(\cdot \,- \theta_{z})}   \bigg ]$$ 
where $\theta_{z}$ is the polar angle of the sampling point $z \in \Omega$ in polar coordinates.\\

We now let the matrix $\textbf{A} \in \C^{N \times N}$ represent the discretized operator $(\Lambda - \Lambda_0)$ and the vector \textbf{b}$_z = \big [ \partial_{\nu(z)} \mathbb{G}\big (  \theta_j ,z \big ) \big ]_{j=1}^{N}$. In our numerical experiments, we add random noise to the discretized operator \textbf{A} such that 
$$ \text{\textbf{A}}^{\delta} = \big [ \text{\textbf{A}}_{i,j} \big( 1 + \delta \text{\textbf{E}}_{i,j} \big) \big ]_{i,j=1}^{N} \quad \text{where} \quad \norm{\text{\textbf{E}}}_{2} = 1.$$ 
Here, the matrix \textbf{E} is taken to have random entries uniformly distributed between $[-1,1]$ and $\delta$ is the relative noise level added to the data in the sense that $\|{\text{\textbf{A}}^{\delta} - \text{\textbf{A}}}\|_{2} \leq \delta \|{\text{\textbf{A}}}\|_{2}$.

When the boundary parameters $\mu$ and $\gamma$ are complex-valued, recall that the we use the imaginary part of the current-gap operator to recover the region $D$. To this end, we denote the matrix $\text{Im} (\textbf{A}^{\delta} )$ as the discretization of the operator $\text{Im}(\Lambda - \Lambda_0)$ with random noise. We now define the discretized imaginary part of the data operator as 
$$\text{Im} (\textbf{A}^{\delta} ) = \frac{1}{2 \text{i}} \big [ \textbf{A}^{\delta} - (\textbf{A}^{\delta})^{*} \big ]$$
Hence, to compute the indicator associated with \thmref{complex-case-thm}, we solve 
$$\text{Im} (\textbf{A}^{\delta} ) \textbf{f}_z = \textbf{b}_z .$$
As specified in \thmref{cgo-comp-id}, the current gap operator is compact which implies that the matrix $\textbf{A}$ is ill-conditioned. Hence, one needs to employ a regularization technique to find an approximate solution to the discretized equation. In our experiments, we use the Spectral cut-off as the regularization scheme and follow a similar procedure demonstrated in \cite{harris2} where \textbf{f}$_z^{\alpha}$ represents the regularized solution to $\text{Im} (\textbf{A}^{\delta} ) \textbf{f}_z = \textbf{b}_z$ and $\alpha > 0$ denotes the regularization parameter. To define the imagining functional, we follow \cite{harris1} to have the following
$$ \big( \textbf{f}_z^{\alpha} , \text{Im}(\textbf{A}^{\delta}) \textbf{f}_z^{\alpha} \big) =  \sum\limits_{j=1}^{N} \frac{\phi^2(\sigma_j  ; \alpha)}{\sigma_j} \big|({\bf u}_j , {\bf b}_z)\big|^2 .$$
Here $\sigma_j$ and ${\bf u}_j$ denotes the singular values and left singular vectors of the matrix $\text{Im}({\textbf{A}}^{\delta})$, respectively. Also, $\phi(t ; \alpha)$ denotes the filter function defined by the regularization scheme used to solve $ \text{Im}(\textbf{A}^{\delta}) \textbf{f}_z = \textbf{b}_z$. The filter function used in our examples is given by 
\begin{align}\label{filters}
 \displaystyle{  \phi(t ; \alpha)= \left\{\begin{array}{lr} 1, &  t^2\geq \alpha,  \\
 				&  \\
 0,&  t^2 < \alpha
 \end{array} \right.}  
\end{align}
which corresponds to Spectral cut-off. Using the above expressions, we can recover the unknown region for the case when the boundary parameters are complex-valued by defining 
 $$W_{\text{reg}} (z) = \big( \textbf{f}^{\alpha}_z , \text{Im}(\textbf{A}^{\delta}) \textbf{f}^{\alpha}_z \big)^{-1}.$$

For the case when the boundary parameters $\mu$ and $\gamma $ are real-valued, then we use discretized operator $\textbf{A}$. By \thmref{gen-DtNprop}, the data operator $(\Lambda - \Lambda_0)$ is compact, which implies that \textbf{A} is ill-conditioned also in this case. We follow a similar procedure from the previous case to define 
$$ \big( \textbf{f}^{\alpha}_z , \textbf{A}^{\delta} \textbf{f}^{\alpha}_z \big) =  \sum\limits_{j=1}^{N} \frac{\phi^2(\sigma_j  ; \alpha)}{\sigma_j} \big|({\bf u}_j , {\bf b}_z)\big|^2 .$$
In this case, $\sigma_j$ and ${\bf u}_j$ denotes the singular values and left singular vectors of the matrix ${\textbf{A}}^{\delta}$, respectively. We also use the same filter function $\phi(t ; \alpha)$ to apply the Spectral cut-off regularization scheme to solve $ \textbf{A}^{\delta} \textbf{f}_z = \textbf{b}_z$. With the above expressions, we can recover the unknown region for the case when the boundary parameters are real valued by defining 
 $$W_{\text{reg}} (z) = \big( \textbf{f}^{\alpha}_z , \textbf{A}^{\delta} \textbf{f}^{\alpha}_z \big)^{-1}$$
In either cases we plot 
$$ \quad W(z) = \bigg \rvert \frac{W_{\text{reg}} (z)}{\|{W_{\text{reg}} (z)}\|_{\infty}} \bigg \rvert^{p}$$
where \thmref{complex-case-thm} and \thmref{real-case-thm} both imply that $W(z) \approx 1$ provided that $z \in D$ as well as $W(z) \approx 0$  provided that $ z \notin D$. In our calculations $p > 0$ is a fixed chosen parameter to sharpen the resolution of the imaging functional. In the following examples  we use the function $W(z)$ to visualize the defective region. \\

\noindent\textbf{Numerical reconstruction of a circular region:}\\
In polar coordinates, we assume $\partial D$ is given by $\rho (\text{cos} ( \theta ), \text{sin} (\theta))$ for some constant $\rho \in (0,1)$. As similarly demonstrated in \cite{EIT-granados1}, since $\Omega$ is taken to be the unit disk in $\mathbb{R}^2$, we make the ansatz that the electrostatic potential $u(r,\theta)$ has the following series representation
\begin{equation}\label{taurus}
u(r , \theta) = a_0 + b_0 \, \text{ln} \, r + \sum_{|n|=1}^{\infty} \left[ a_n r^{|n|} + b_n r^{-|n|} \right] \text{e}^{\text{i}n \theta} \quad \text{in} \quad \Omega \backslash D
\end{equation}
whereas
$$u(r , \theta) = c_0 + \sum_{|n|=1}^{\infty} c_n r^{|n|} \text{e}^{\text{i}n \theta} \quad \text{in} \quad D.$$
Note, that the electrostatic potential $u(r, \theta )$ is harmonic in both the annular and circular regions which are separated by the interior boundary $\partial D$.

Recall, that the boundary parameters $\gamma$ and $\mu$ satisfy \eqref{gamma-bounds} and \eqref{mu-bounds}, respectively. Furthermore, for simplicity, we assume that $\gamma$ and $\mu$ are constant. Thus, we are able to determine the Fourier coefficients $a_n$ and $b_n$ by using the boundary conditions at $r=1$ and $r = \rho$ given by
$$u(1, \theta ) = f(\theta), \quad u^{+} (\rho , \theta ) = u^{-} ( \rho , \theta ),$$  $$\quad \text{and} \quad \partial_{r} u^{+} (\rho , \theta ) - \partial_{r} u^{-} (\rho , \theta ) = \bigg( - \dfrac{\mu}{\rho^2} \dfrac{\partial^2}{\partial \theta^2} + \gamma \bigg) u(\rho , \theta). $$
We let $f_n$ for $n \in \mathbb{Z}$ denote the Fourier coefficients for the voltage $f$. Note, that the boundary condition at $r=1$ above gives that $$ a_0 = f_0 \quad \text{and} \quad a_n + b_n = f_n \quad \text{for all} \enspace n \neq 0.$$
The first boundary conditions at $r = \rho$ give that  
$$ b_0 = \frac{\gamma \rho}{1 - \gamma \rho \text{ln} \, \rho} f_0 \quad \text{and} \quad b_n = \rho^{2|n|} (c_n  - a_n ). $$ Using the generalized Robin transmission condition, and after some calculations we get that 
$$a_n = \left[ \frac{ \gamma \rho^{2|n|} + 2|n| \rho^{|n|}+ \mu |n|^2}{ \mu |n|^2 + 2|n| \rho^{|n|} + \gamma \rho^{2|n|} (1 - \rho^{2|n|}) - \mu |n|^2 \rho^{2|n|}  }  \right] f_n$$ 
and  $$b_n = \left[ \frac{- \rho^{2|n|} (\gamma \rho^{2|n|} + \mu |n|^2)}{ \mu |n|^2 + 2|n| \rho^{|n|} + \gamma \rho^{2|n|} (1 - \rho^{2|n|}) - \mu |n|^2 \rho^{2|n|}} \right] f_n \quad \text{for all} \enspace  n \neq 0.$$ 
Plugging the sequences into \eqref{taurus} gives that the corresponding current on the boundary of the unit disk is given by
\begin{equation}\label{unormal}
\partial_{r} u(1 , \theta ) = \sigma_0 f_0 + \sum_{|n|=1}^{\infty} |n| \sigma_n f_n \text{e}^{\text{i}n \theta}
\end{equation}
where 
$$\sigma_0 = \frac{\gamma \rho}{1 - \gamma \rho \text{ln} \, \rho} \quad \text{and} \quad \sigma_n = \dfrac{ 2|n| \rho^{|n|} + (\mu |n|^2 + \gamma \rho^{2|n|}) (1 + \rho^{2|n|})}{2|n| \rho^{|n|} + (\mu |n|^2 + \gamma \rho^{2|n|})(1 - \rho^{2|n|})} \quad \text{for all} \quad n \neq 0.$$
It is clear that the electrostatic potential and subsequent current for the material without a defective region is given by 
\begin{equation}\label{u0normal}
u_0 (r , \theta ) = f_0 + \sum_{|n|=1}^{\infty} f_n r^{|n|} \text{e}^{\text{i}n \theta} \quad \text{and} \quad \partial_{r} u_0 (1 , \theta ) = \sum_{|n|=1}^{\infty} |n| f_n \text{e}^{\text{i}n \theta}.
\end{equation}
Subtracting equation \eqref{u0normal} from \eqref{unormal} gives a series representation of the current gap operator. By interchanging summation with integration we obtain 
$$ (\Lambda - \Lambda_0)f = \frac{1}{2 \pi} \int_{0}^{2 \pi} K (\theta , \phi) f ( \phi) \, \text{d} \phi \quad \text{where} \quad K(\theta , \phi ) = \sigma_0 + \sum_{|n|=1}^{\infty} |n| (\sigma_n - 1) \text{e}^{\text{i}n (\theta - \phi)}.$$

This representation allows one to easily construct synthetic data for numerical experiments. We now introduce a theorem regarding the convergence of the truncated series approximation for the above integral operator.

\begin{theorem}\label{convergence}
Let $(\Lambda - \Lambda_0)_N : H^{1/2} ( 0 , 2 \pi) \rightarrow H^{-1/2} (0 , 2 \pi)$ be the truncated series approximation of $(\Lambda - \Lambda_0)$. Then we have that in the operator norm 
$$ \|{(\Lambda - \Lambda_0) - (\Lambda - \Lambda_0)_N}\| \leq C \rho^{2(N+1)}$$
where $C >0$ is independent of $N$.
\end{theorem}
\begin{proof}
To prove the claim, consider ${\displaystyle \big[(\Lambda - \Lambda_0) - (\Lambda - \Lambda_0)_N\big] f = \sum_{|n|= N+1}^{\infty} |n| f_n (\sigma_n - 1) \text{e}^{\text{i}n \theta}. }$
By the Cauchy-Schwarz inequality in $\ell^2$ we have that
\begin{align*}
	\Big | \big[(\Lambda - \Lambda_0) - (\Lambda - \Lambda_0)_N\big] f \Big |^2   &\leq  \bigg ( \sum_{|n|= N+1}^{\infty} |\sigma_n - 1|^2 |n| |\text{e}^{\text{i}n \theta}|^2 \bigg ) \bigg ( \sum_{|n|= N+1}^{\infty} |n| |f_n|^2 \bigg ) \\
	&\leq \norm{f}_{H^{1/2}(0,2 \pi)}^{2} \bigg ( \sum_{|n|= N+1}^{\infty} |\sigma_n - 1|^2 |n| \bigg )
\end{align*} 
After some calculations, we obtain that $|\sigma - 1|^2 |n| \leq C_{\mu , \gamma, \rho} |n| \rho^{4 |n|}$, where $C_{\mu , \gamma, \rho}$ is a positive constant that depends on $\mu , \gamma$, and $\rho$, but independent of $n$ and $N$. This gives that $$\norm{[(\Lambda - \Lambda_0) - (\Lambda - \Lambda_0)_N] f }_{\infty} \leq C_{\mu , \gamma, \rho} \norm{f}_{H^{1/2}(0,2 \pi)} \rho^{2(N+1)}. $$
We obtain our result by using the fact that the $H^{-1/2} (0, 2 \pi)$--norm is bounded by the $L^{\infty} ( 0 , 2 \pi )$--norm.
\end{proof}
\thmref{convergence} demonstrates that the convergence for the approximation is geometric. Thus, we do not need many terms in the kernel function to approximate the data operator in order to obtain desirable results. In the following examples, we approximate the kernel function $K(\theta , \phi )$ given above  by truncating the series for $|n|=1, \hdots , 10$. With this, we then discretize the truncated integral operator by a 64 equally spaced grid on $[0 ,2 \pi)$ using a collocation method. \\

\noindent{\bf Example 1: complex coefficients}\\
For numerical reconstructions here, we set the decay parameter $p = 1$. In Figure \ref{complex_circle_v1}, we take $\rho = 0.2$ and $\delta = 0.05$ which corresponds to $5 \% $ relative random noise added to the data. The boundary coefficients are $\gamma = 2 - 0.5  \text{i}$ and $\mu = 0.1 - \text{i}$. Here the Spectral cut-off regularization parameter is taken to be $\alpha=10^{-17}$. The dotted lines are the boundaries of $\partial \Omega$ and $\partial D$ with the solid line being the approximation via the level curve, which is chosen where the contour plot goes from red to black.
\begin{figure}[!h]
\centering 
\includegraphics[scale=0.17]{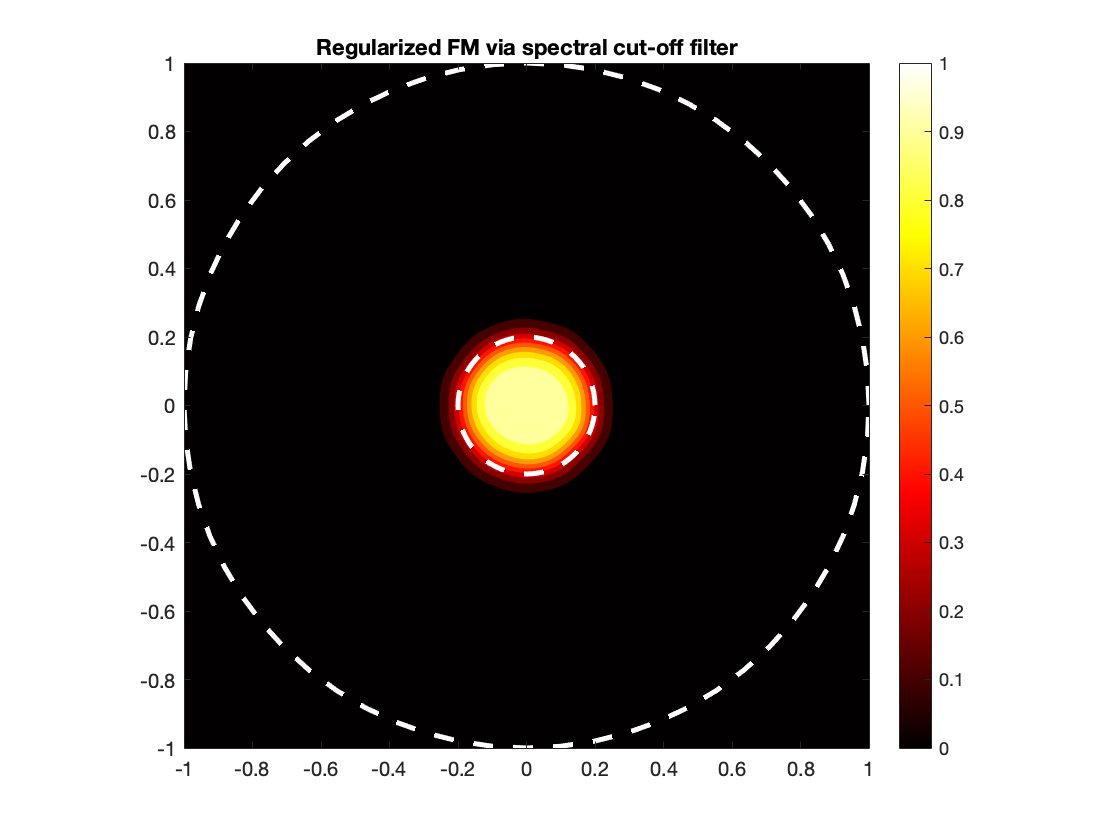} \hspace{-.5in}
\includegraphics[scale=0.17]{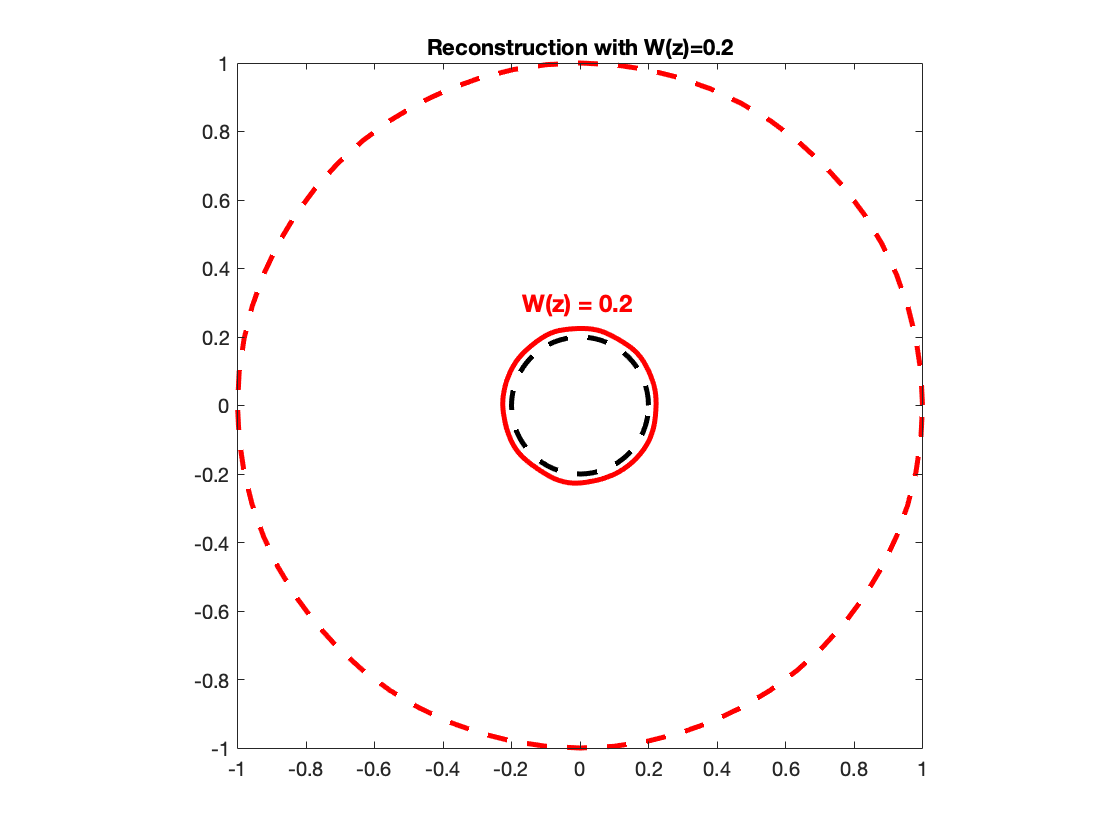}
\caption{Reconstruction of a circular region with $\rho=0.2$ via the regularized factorization method. Boundary coefficients are $\gamma = 2 - 0.5 \text{i}$ and $\mu = 0.1 -  \text{i}$. Contour plot of $W(z)$ on the left and level curve when $W(z)= 0.2$ on the right.}
\label{complex_circle_v1}
\end{figure}

In Figure \ref{complex_circle_v2}, we take $\rho = 0.7$ and increase error to $\delta = 0.1$ which corresponds to $10 \% $ relative random noise added to the data. The boundary coefficients are taken to be $\gamma = 2-3 \text{i}$ and $\mu = 1 - 4 \text{i}$. Here the Spectral cut-off regularization parameter is set to $\alpha=10^{-4}$. The dotted lines are the boundaries of $\partial \Omega$ and $\partial D$ with the solid line being the approximation via the level curve.\\
\begin{figure}[!h]
\centering 
\includegraphics[scale=0.17]{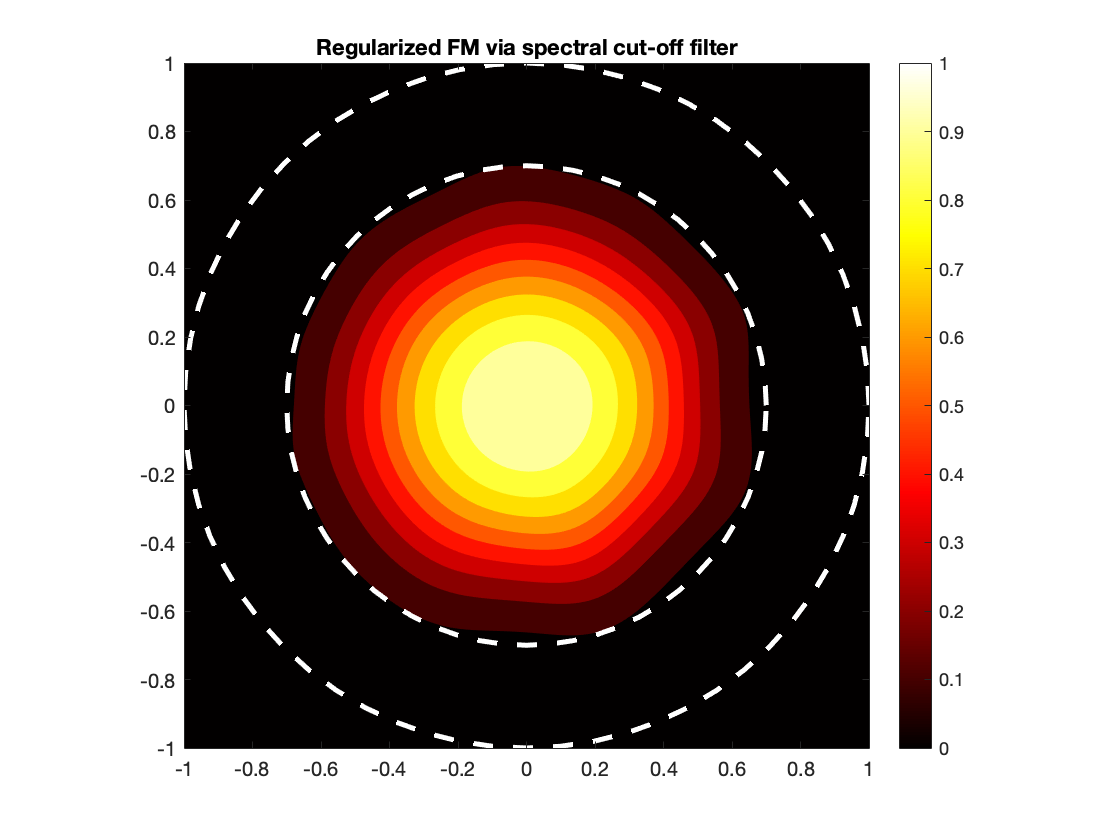} \hspace{-.5in}
\includegraphics[scale=0.17]{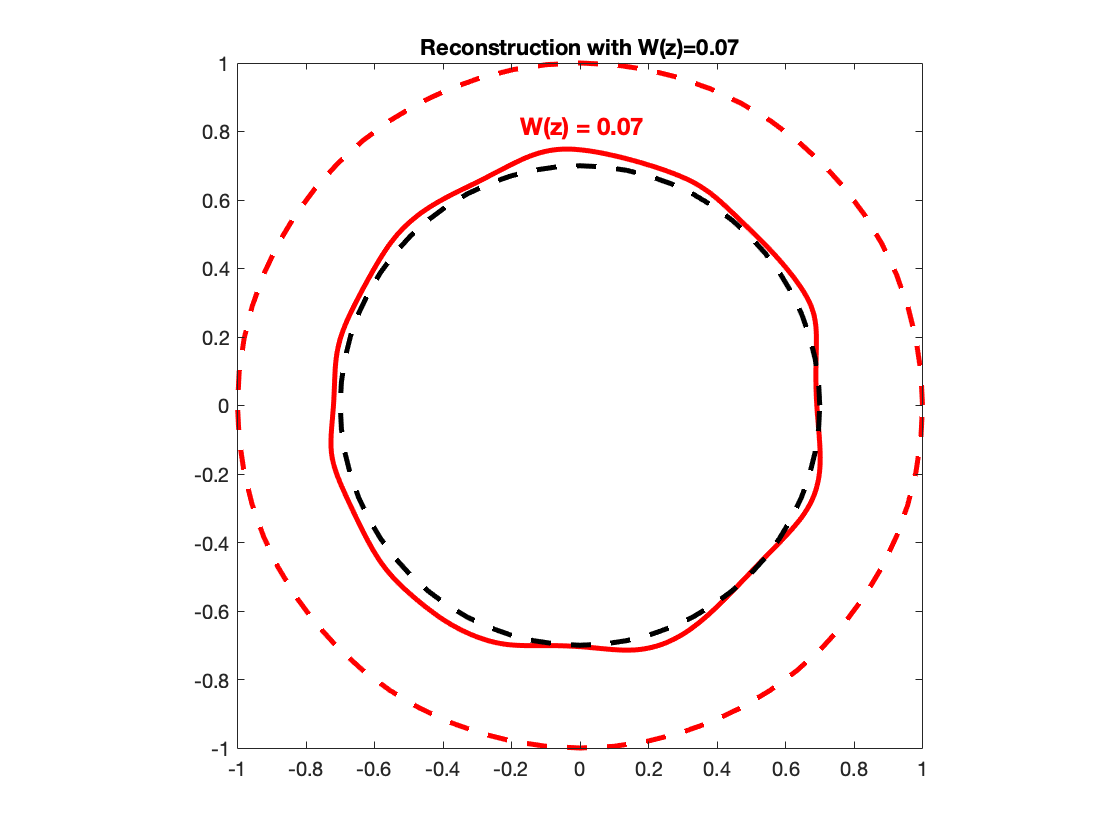}\caption{Reconstruction of a circular region with $\rho=0.7$ via the regularized factorization method. Boundary coefficients are $\gamma = 2-3 \text{i}$ and $\mu = 1 - 4 \text{i}$. Contour plot of $W(z)$ on the left and level curve when $W(z)= 0.2$ on the right.}
\label{complex_circle_v2}
\end{figure}

\noindent{\bf Example 2: real coefficients}\\
For numerical reconstructions here, we let the decay parameter $p = 4$. In the following numerical experiments, we consider cases where the boundary coefficients of $\partial D$ are strictly real-valued. We also continue using the Spectral cut-off regularization scheme. In Figure \ref{real_circle_v1}, we take $\rho = 0.25$ and $\delta = 0.05$ which corresponds to $5 \% $ relative random noise added to the data. The Spectral cut-off regularization parameter $\alpha=10^{-15}$. The dotted lines are the boundaries of $\partial \Omega$ and $\partial D$ with the solid line is the approximation via the level curve.
\begin{figure}[!h]
\centering 
\includegraphics[scale=0.17]{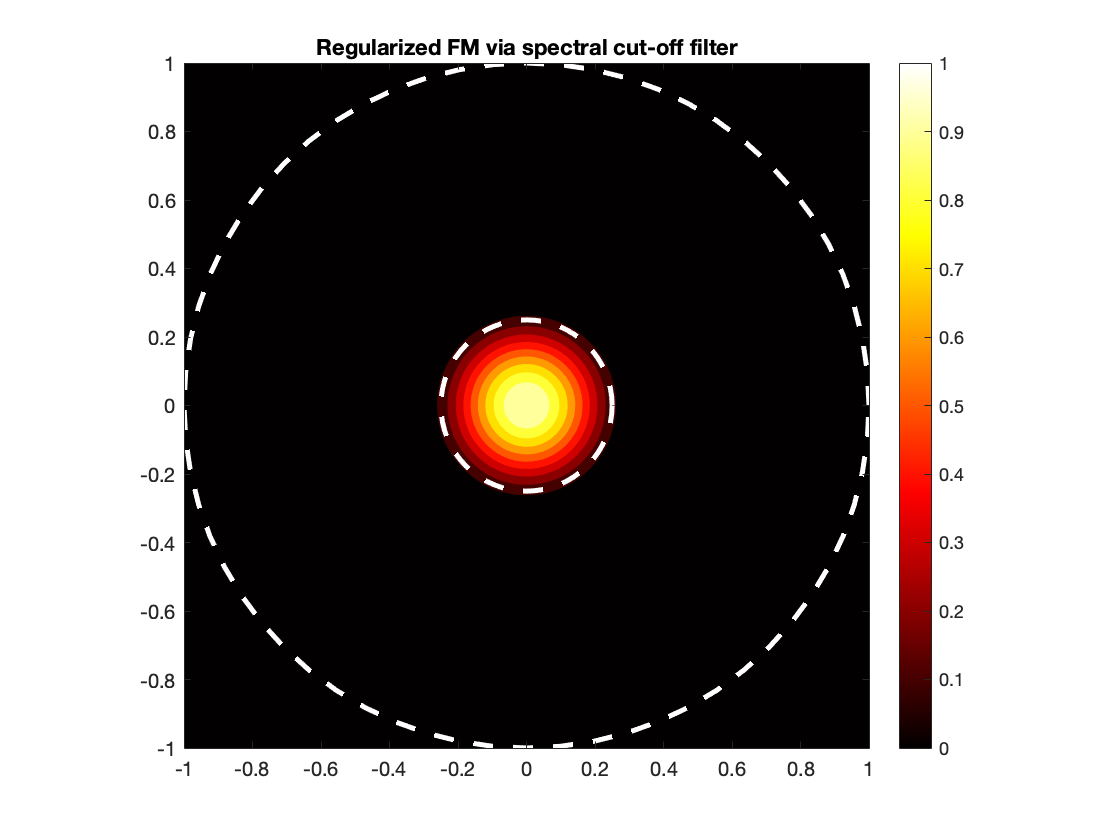} \hspace{-.5in}
\includegraphics[scale=0.17]{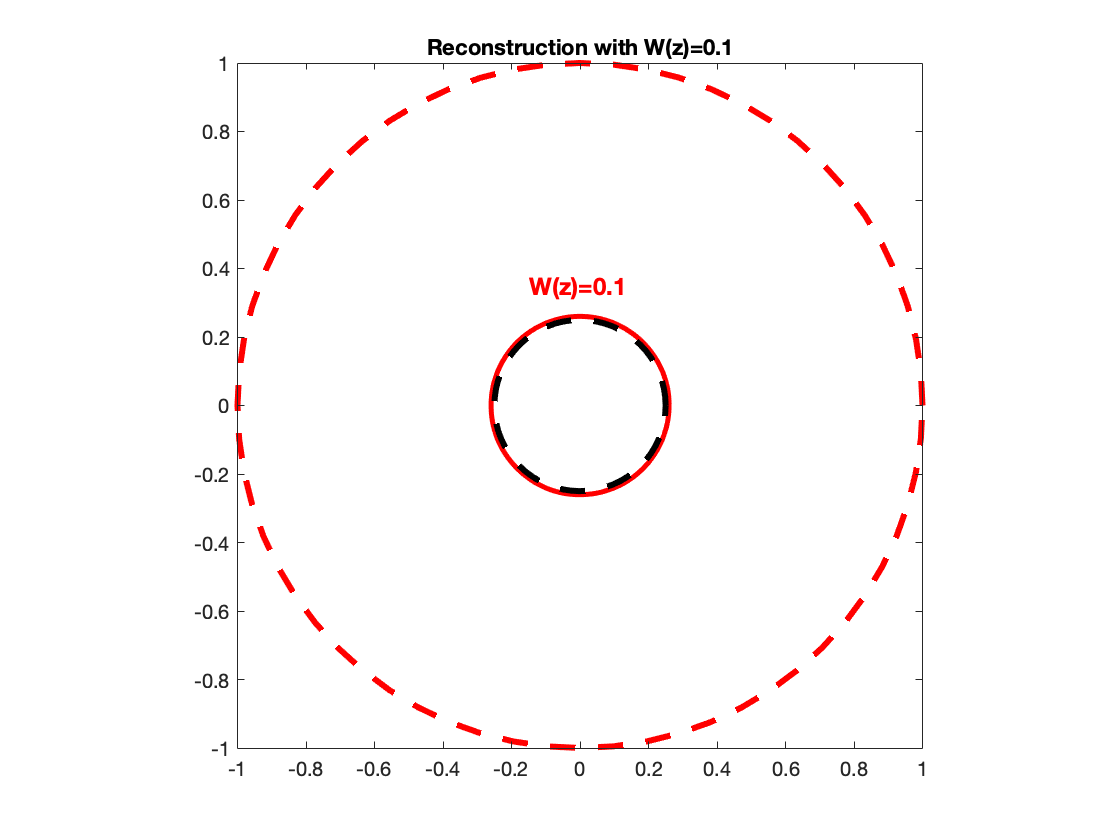}\caption{Reconstruction of a circular region with $\rho=0.25$ via the regularized factorization method. Boundary coefficients are $\gamma = 1.2$ and $\mu = 0.5$. Contour plot of $W(z)$ on the left and level curve when $W(z)= 0.1$ on the right.}
\label{real_circle_v1}
\end{figure}

In Figure \ref{real_circle_v2}, we take $\rho = 0.75$ and $\delta = 0.1$ which corresponds to $10 \% $ relative random noise added to the data. Here, the Spectral cut-off regularization parameter is $\alpha=10^{-5}$. Once again, the dotted lines are the boundaries of $\partial \Omega$ and $\partial D$.\\
\begin{figure}[!h]
\centering 
\includegraphics[scale=0.17]{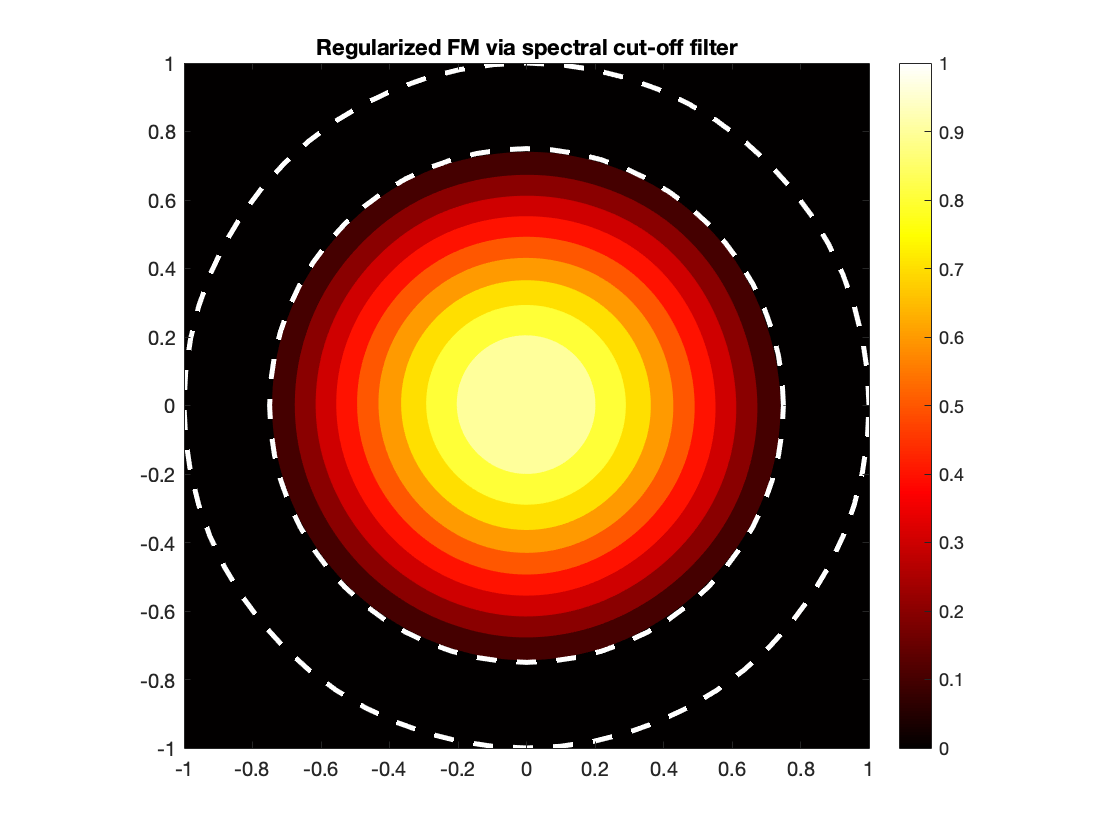} \hspace{-.5in}
\includegraphics[scale=0.17]{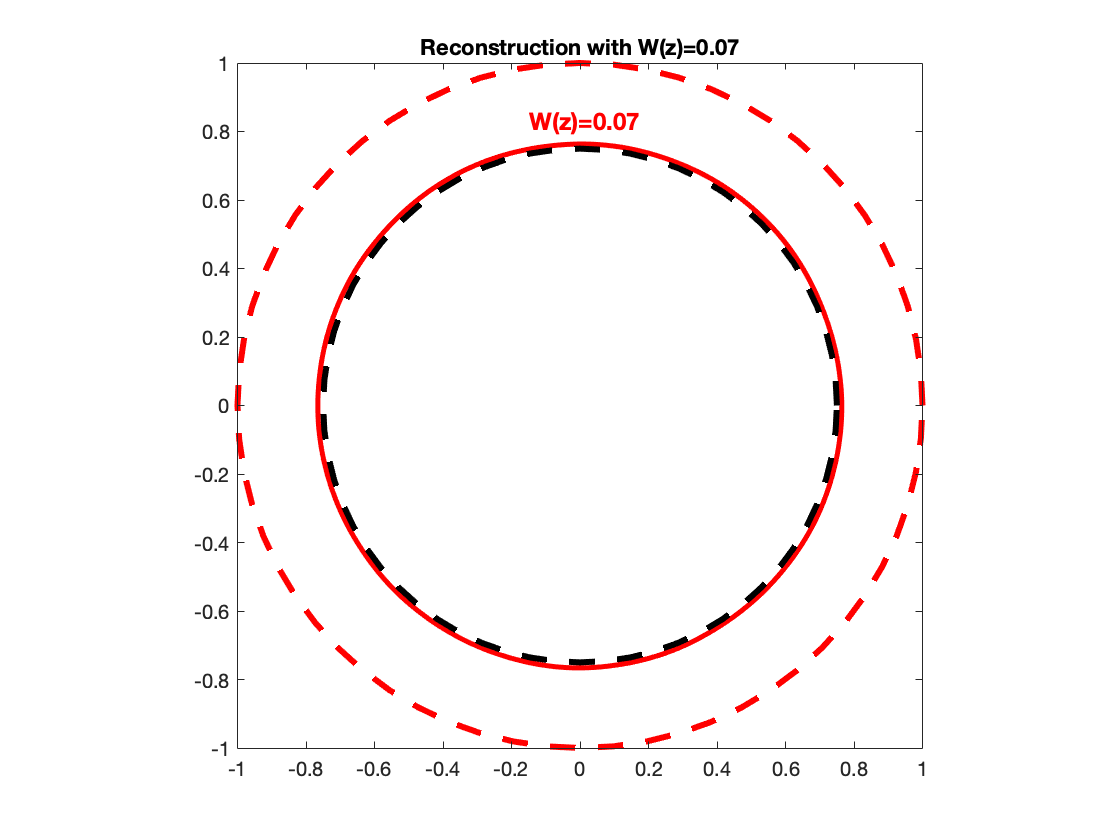}\caption{Reconstruction of a circular region with $\rho=0.75$ via the regularized factorization method. Boundary coefficients are $\gamma = 0.6$ and $\mu = 1.6$. Contour plot of $W(z)$ on the left and level curve when $W(z)= 0.07$ on the right.}
\label{real_circle_v2}
\end{figure}

\section{Conclusions}\label{conclusion}
In this paper, we have studied the Regularized Factorization Method for recovering an inclusion from electrostatic data. The factorization of the data operator depends on whether the interior boundary parameters are complex or real valued. Since we employ a Qualitative Method instead of an iterative method we do not require a priori knowledge about the region of interest or boundary condition. We reduced the regularity assumptions from previous works by requiring the full knowledge of the DtN mapping. We note that the analysis provided here can be used to study this inverse shape problem in $\R^d$ for $d = 2,3$. Our algorithm allows for fast and accurate reconstruction with little a priori knowledge of the region of interest $D$. We also showed the uniqueness of the inverse impedance problem. A future direction for this project can be to study the inverse parameter problem and derive a non-iterative method for recovering the boundary coefficients $\mu$ and $\gamma$. One could also study the error stability as our numerical experiments suggest that our algorithm is stable with respect to noise in Cauchy data. Lastly, one could also consider studying the direct sampling method(see for e.g. \cite{DSM-DOT, DSM-EIT, DSM-nf}) for this problem. \\

\noindent{\bf Acknowledgments:} The research of G. Granados, I. Harris and H. Lee is partially supported by the NSF DMS Grant 2107891.



\begin{thebibliography}{99}











\bibitem{arens} 
\newblock T. Arens, 
\newblock Why linear sampling method works.
\newblock {\it Inverse Problems} {\bf 20} (2004), 163--173.


\bibitem{GLSM}
\newblock L. Audibert and H. Haddar, 
\newblock A generalized formulation of the linear sampling method with exact characterization of targets in terms of far field measurements.
\newblock {\it Inverse Problems}, {\bf 30}, (2014), 035011.





\bibitem{eit-review}
\newblock L. Borcea,
\newblock Electrical impedance tomography.
\newblock {\it Inverse Problems}, {\bf18}, (2002) R99--R136



\bibitem{eit-review-amend}
\newblock L. Borcea,
\newblock Addendum to: Electrical impedance tomography.
\newblock {\it Inverse Problems}, {\bf19}, (2003) 997--998



\bibitem{EM-cakoni}
\newblock F. Cakoni, I. de Teresa, and P. Monk
\newblock{Nondestructive testing of delaminated interfaces between two materials using electromagnetic interrogation}.
\newblock{\it Inverse Problems}, {\bf 34}, (2018) 065005


\bibitem{CHK-gibc}
\newblock F. Cakoni , Y. Hu and R. Kress, 
\newblock Simultaneous reconstruction of shape and generalized impedance functions in electrostatic imaging, 
\newblock {\it Inverse Problems} {\bf 30} (2014) 105009.


\bibitem{Heejin1}
\newblock F. Cakoni, H. Lee,  P. Monk and Y. Zhang
\newblock{A spectral target signature for thin surfaces with higher order jump conditions}.
\newblock{\it Inverse Problems and Imaging}, {\bf 16(6)}, (2022) 1473--1500.










\bibitem{Haddar1}
\newblock S. Chaabane, B. Charfi and H. Haddar, 
\newblock Reconstruction of discontinuous parameters in a second order impedance boundary operator. 
\newblock{\it Inverse Problems} \textbf{32(10)}, (2016) 105004



\bibitem{FM-GIBC}
\newblock M. Chamaillard, N. Chaulet and H. Haddar, 
\newblock Analysis of the factorization method for a general class of boundary conditions, 
\newblock {\it Journal of Inverse and Ill-posed Problems} {\bf 22} No. 5 (2014) 643-670.












\bibitem{EIT-cheney}
\newblock M. Cheney, D. Isaacson and J.-C. Newell, 
\newblock Electrical impedance tomography.
\newblock {\it SIAM Rev.}, {\bf 41}, (1999), 85--101.



\bibitem{DSM-DOT}
\newblock Y.T. Chow, K. Ito, K. Liu and J. Zou,
\newblock { Direct Sampling Method for Diffusive Optical Tomography}.
\newblock {\it SIAM J. Sci. Comput.}, {\bf 37:4}, (2015), A1658--A1684.


\bibitem{DSM-EIT}
\newblock Y.T. Chow, K. Ito, K. Liu and J. Zou,
\newblock { Direct Sampling Method for Electrical Impedance Tomography}.
\newblock {\it Inverse Problems}, {\bf 30}, (2014), 095003.


\bibitem{colton1}
\newblock D. Colton and A. Kirsch
\newblock { A simple method for solving inverse scattering problems in the resonance region}.
\newblock {\it Inverse Problems}, {\bf 12}, (1996), 383-393.


\bibitem{deng1}
\newblock Y. Deng and X. Liu
\newblock {Electromagnetic imaging methods for nondestructive evaluation applications}.
\newblock {\it Sensors}, {\bf 11}, (2011), 11774--808.


\bibitem{embry}
\newblock M.R. Embry,
\newblock {Factorization of operators on Banach space}.
\newblock {\it Proc. Amer. Math. Soc.}, {\bf 38}, (1973), 587--590.

\bibitem{evans}
\newblock L. Evans, 
\newblock \emph{``Partial Differential Equation''},  
\newblock 2nd edition, AMS Providence RI, 2010.


\bibitem{franchois1}
\newblock A. Franchois and C. Pichot, 
\newblock {Microwave imaging-complex permittivity reconstruction with a Levenberg--Marquardt method}.
\newblock {\it IEEE Trans. Antennas Propag.},  {\bf 45}, (1997), 203--215


\bibitem{EIT-FM}
\newblock B. Gebauer and N. Hyv\"{o}nen, 
\newblock Factorization method and irregular inclusions in electrical impedance tomography.
\newblock {\it Inverse Problems}, {\bf 23}, (2007),  2159--2170 




\bibitem{EIT-granados1} 
\newblock G. Granados and I. Harris,
\newblock Reconstruction of small and extended regions in EIT with a Robin transmission condition.
\newblock {\it Inverse Problems}, {\bf 38}, (2022), 105009




\bibitem{MUSIC-Hanke2} 
\newblock M. Hanke and M. Br\"{u}hl,
\newblock Recent Progress in Electrical Impedance Tomography.
\newblock {\it Inverse Problems}, {\bf 19}, (2003), 1--26.


\bibitem{eit-harrach1} 
\newblock B. Harrach,
\newblock Recent progress on the factorization method for electrical impedance tomography.
\newblock {\it Comput. Math. Methods Med.}, (2013), 425184.


\bibitem{eit-transmission2} 
\newblock B. Harrach,
\newblock Uniqueness, stability and global convergence for a discrete inverse elliptic Robin transmission problem.
\newblock {\it Numer. Math.}, {\bf 147} (2021) 29--70.


\bibitem{eit-transmission1} 
\newblock B. Harrach and H. Meftahi,
\newblock Global Uniqueness and Lipschitz-Stability for the Inverse Robin Transmission Problem.
\newblock {\it SIAM J. App. Math.}, {\bf 79:2} (2019) 525--550.


\bibitem{harris1}
\newblock I. Harris, 
\newblock Regularization of the Factorization Method applied to diffuse optical tomography. 
\newblock {\it Inverse Problems}, {\bf 37}, (2021), 125010.


\bibitem{harris2}
\newblock I. Harris, 
\newblock Detecting inclusions with a generalized impedance condition from electrostatic data via sampling. 
\newblock {\it Math. Methods Appl. Sci.},  {\bf 49:18} (2019), 6741--6756. 


\bibitem{regfm2}
\newblock I. Harris, 
\newblock Regularized factorization method for a perturbed positive compact operator applied to inverse scattering,
\newblock {\it  Inverse Problems}, {\bf 39}, (2023),  115007.



\bibitem{holmgren}
\newblock H.  Hedenmalm, 
\newblock On the uniqueness theorem of Holmgren.
\newblock {\it Math. Z.},  {\bf 281}, (2015) 357--378. 




\bibitem{kharkovsky1}
\newblock S. Kharkovsky and R. Zoughi,
\newblock Microwave and millimeter wave nondestructive testing and evaluation—overview and recent advances.
\newblock {\it IEEE Instrum. Meas. Mag.}, {\bf 10}, (2007), 26--38. 


\bibitem{kirschbook}
\newblock A. Kirsch A and N. Grinberg, 
\newblock ``The Factorization Method for Inverse Problems''.
\newblock 1st edition Oxford University Press, Oxford 2008.








\bibitem{JIIP} 
\newblock A. Laurain and H. Meftahi,
\newblock Shape and parameter reconstruction for the Robin transmission inverse problem.
\newblock {\it J. Inverse Ill-Posed Probl.}, {\bf 24:6} (2016) 643--662.


\bibitem{DSM-nf} 
\newblock X. Liu, S. Meng and B. Zhang,
\newblock Modified sampling method with near field measurements.
\newblock {\it SIAM J. App. Math.}, {\bf 82:1} (2022) 244--266.


\bibitem{mclean-book} 
\newblock W. McLean
\newblock ``Strongly elliptic systems and boundary integral equations'',
\newblock Cambridge: Cambridge University Press 2000.


\bibitem{mueller-book} 
\newblock J. Mueller  and S. Siltanen
\newblock ``Linear and Nonlinear Inverse Problems with Practical Applications'',
\newblock 1st edition, SIAM Philadelphia PA, 2012. 








\bibitem{tamori1} 
\newblock Y. Tamori, E. Suzuki, and W.M. Deng, 
\newblock {Epithelial tumors originate in tumor hotspots, a tissue--intrinsic microenvironment}.
\newblock {\it PLoS Biol.}, {\bf 14}, (2016), e1002537.



\end{thebibliography}
\end{document}